\documentclass[12pt]{amsart}
\usepackage[usenames]{color}
\usepackage{amsfonts,amssymb,amsrefs,comment,url,amscd,graphicx} 
\usepackage[all]{xy}
\usepackage{a4wide}
\usepackage{hyperref}

\newcommand{\Irr}{\operatorname{Irr}}
\newcommand{\Cusp}{\Irr_c}
\newcommand{\leunq}{\trianglelefteq}
\newcommand{\xbasis}{x}
\newcommand{\ssf}{\operatorname{ss}}
\newcommand{\one}{\mathbf{1}}
\newcommand{\cnt}{\mathbf{C}}
\newcommand{\stb}{\mathbf{G}}
\newcommand{\repn}{\mathbf{A}}
\newcommand{\Vect}{\mathbf{V}}
\newcommand{\grdim}{\operatorname{grdim}}
\newcommand{\pr}{\operatorname{pr}}
\newcommand{\canJ}{\mathfrak{p}}
\newcommand{\IrrS}{\Irr^\square}
\newcommand{\n}{\mathfrak{n}}
\newcommand{\embd}{\imath}
\newcommand{\srjct}{\jmath}
\newcommand{\mult}{\mu}
\newcommand{\rk}{\operatorname{rk}}

\newcommand{\Aut}{\mathbb{G}}
\newcommand{\LeviAut}{\mathbb{M}}
\newcommand{\ParAut}{\mathbb{P}^{\embd}}
\newcommand{\ParAutt}{\mathbb{P}^{\embd'}}
\newcommand{\UnipAut}{\mathbb{U}^{\embd}}
\newcommand{\UnipAutt}{\mathbb{U}^{\embd'}}
\newcommand{\levistb}{\mathbf{M}}
\newcommand{\parstb}{\mathbf{P}^{\embd}}
\newcommand{\parstbb}{\mathbf{P}^{\embd'}}
\newcommand{\unipstb}{\mathbf{U}^{\embd}}
\newcommand{\unipstbb}{\mathbf{U}^{\embd'}}
\newcommand{\grend}{\mathbb{E}}
\newcommand{\grendpres}{\grend^{\embd}}
\newcommand{\grendpress}{\grend^{\embd'}}
\newcommand{\cntpres}{\cnt_{\n}^{\embd}}
\newcommand{\cntpress}{\cnt_{\n}^{\embd'}}
\newcommand{\cntbw}{\cnt_{\n}^{\embd,\embd'}}
\newcommand{\Specvar}{\mathfrak{Y}}
\newcommand{\maxdel}{\Delta}
\newcommand{\mxprt}{{\operatorname{mx}}}
\newcommand{\nmxprt}{{\operatorname{nmx}}}
\newcommand{\cs}{\mathfrak{s}}
\newcommand{\comp}{\mathfrak{C}}
\newcommand{\Prodcomp}{\mathfrak{Z}}
\newcommand{\JH}{\operatorname{JH}}
\newcommand{\Hom}{\operatorname{Hom}}
\newcommand{\Gr}{\mathcal{G}}
\newcommand{\lderiv}[2]{\,_{#1}{#2}}
\newcommand{\OO}{\mathcal{O}}

\newcommand{\commvar}{\mathfrak{X}}
\newcommand{\cl}[1]{\overline{#1}}                            
\newcommand{\Lie}{\operatorname{Lie}}
\newcommand{\N}{\mathbb{N}}
\newcommand{\C}{\mathbb{C}}

\newcommand{\Z}{\mathbb{Z}}

\newcommand{\id}{\operatorname{id}}                            
\newcommand{\End}{\operatorname{End}}

\newcommand{\Seg}{\mathcal{SEG}}
\newcommand{\std}[1]{\mathfrak{z}(#1)}
\newcommand{\cstd}[1]{\lambda(#1)}
\newcommand{\MS}[1]{\N({#1})}
\newcommand{\Reps}{\mathcal{R}}
\newcommand{\abs}[1]{\left|{#1}\right|}
\newcommand{\rest}{\big|}
\newcommand{\dembed}{\Sigma_{\embd,\embd'}}
\newcommand{\tr}{\operatorname{tr}}
\newcommand{\GL}{\operatorname{GL}}
\newcommand{\supp}{\operatorname{supp}}
\newcommand{\m}{\mathfrak{m}}
\newcommand{\soc}{\operatorname{soc}}                          
\newcommand{\SI}{SI}                                           
\newcommand{\lshft}[1]{\overset{\leftarrow}{#1}\vphantom{#1}}               
\newcommand{\rshft}[1]{\overset{\rightarrow}{#1}\vphantom{#1}}              
\newcommand{\rlshft}[1]{\overset{\leftrightarrow}{#1}\vphantom{#1}}              
\newcommand{\LI}{\operatorname{SA}}                                          
\newcommand{\SLI}{\operatorname{SSA}}                                          
\newcommand{\LC}{\operatorname{SG}}                                          
\newcommand{\IC}{\operatorname{IG}}
\newcommand{\GLS}{GLS}
\newcommand{\Lieg}{\mathfrak{g}}
\renewcommand{\Im}{\operatorname{Im}}
\newcommand{\sprt}
{\mathrel{\reflectbox{\rotatebox[origin=c]{315}{$\pitchfork$}}}}

\newtheorem*{theorem}{Theorem}
\newtheorem*{lemma}{Lemma}
\newtheorem*{conjecture}{Conjecture}
\newtheorem*{proposition}{Proposition}
\newtheorem*{corollary}{Corollary}
\newtheorem*{example}{Example}
\newtheorem*{definition}{Definition}
\newtheorem*{remark}{Remark}
\newtheorem*{question}{Question}

\numberwithin{equation}{section}

\newcommand{\Erez}[1]{{\color{magenta}{#1}}}

\begin{document}

\title[Conjectures and results about parabolic induction]
{Conjectures and results about parabolic induction of representations of $\GL_n(F)$}

\author{Erez Lapid}
\address{Department of Mathematics, Weizmann Institute of Science, Rehovot 7610001, Israel}
\email{erez.m.lapid@gmail.com}
\author{Alberto M\'inguez}
\address{Department of Mathematics, University of Vienna, Oskar-Morgenstern-Platz 1, 1090 Wien, Austria}
\email{alberto.minguez@univie.ac.at}
\thanks{A.M. was partially funded by grant P12-FQM-2696.}
\date{}
\maketitle

\begin{abstract}
In 1980 Zelevinsky introduced commuting varieties whose irreducible components classify
complex, irreducible representations of the general linear group over a non-archimedean local field
with a given supercuspidal support.
We formulate geometric conditions for certain triples of such components and
conjecture that these conditions are related to irreducibility of parabolic induction.
The conditions are in the spirit of the Geiss--Leclerc--Schr\"oer condition that occurs
in the conjectural characterization of $\square$-irreducible representations.
We verify some special cases of the new conjecture and check that the geometric and representation-theoretic conditions are
compatible in various ways.
\end{abstract}

\setcounter{tocdepth}{1}
\tableofcontents

\section{Introduction}
\subsection{}
Let $F$ be a local non-archimedean field.
The smooth, complex representations of $\GL_n(F)$, $n\ge0$
were studied in depth in the seminal work of Bernstein and Zelevinsky \cites{MR0425031, MR0579172, MR584084}.
In particular, Zelevinsky obtained a classification of the set $\Irr=\cup\Irr_n$ of irreducible
representations of $\GL_n(F)$, $n\ge0$ in terms of multisegments -- an essentially combinatorial object.
Denote by $Z(\m)$ the irreducible representation corresponding to a multisegment $\m$.
A basic property is that for any $\pi=Z(\m), \pi'=Z(\m')\in\Irr$, the representation
$Z(\m+\m')$ occurs with multiplicity one in the Jordan--H\"older sequence of the representation
$\pi\times\pi'$ parabolically induced from $\pi\otimes\pi'$.
In particular, if $\pi\times\pi'$ happens to be irreducible, then it is equivalent to $Z(\m+\m')$.
However, the problem of characterizing the irreducibility of $Z(\m)\times Z(\m')$ was left open.

An important special case is when $\pi'=\pi$. We say that a representation $\pi$ is $\square$-irreducible if
$\pi\times\pi$ is irreducible.
It is a well-known consequence of a theorem of Bernstein \cite{MR748505} that every \emph{unitarizable} representation
$\pi\in\Irr$ is $\square$-irreducible. However, our emphasis in this paper is on \emph{general} irreducible representations.
The first example of a non-$\square$-irreducible $\pi\in\Irr$ was given by Leclerc \cite{MR1959765}.

The analogues of $\square$-irreducible representations in different (but related) contexts (where a different terminology is used)
play an important role in the monoidal categorification of cluster algebras, a problem considered
by Hernandez--Leclerc \cites{MR2682185, MR3077685} and more recently by  Kang--Kashiwara--Kim--Oh \cite{MR3758148}.
We will not say anything about this problem here except to mention that
the argument of \cite{MR3314831}, adapted to the case at hand, shows that
if $\pi$ is $\square$-irreducible, then for any $\pi'\in\Irr$,
the socles $\soc(\pi\times\pi')$ and $\soc(\pi'\times\pi)$ are irreducible and each occurs with multiplicity one in
the Jordan--H\"older sequence of $\pi\times\pi'$ \cite{MR3866895}.
It follows that if at least one of $\pi=Z(\m)$ and $\pi'=Z(\m')$ is $\square$-irreducible, then the irreducibility
of $Z(\m)\times Z(\m')$ is equivalent to the conjunction of the condition
\begin{equation} \label{eq: LI}
Z(\m+\m')\hookrightarrow Z(\m)\times Z(\m')
\end{equation}
and its symmetric counterpart (interchanging $\m$ and $\m'$).

On the other hand, Zelevinsky proposed a geometric framework for his classification \cites{MR617466, MR783619} and
in this context, the work of Geiss--Leclerc--Schr\"oer \cite{MR2144987} suggests a simple, conjectural geometric criterion for
the $\square$-irreducibility of $Z(\m)$.
More precisely, fix a finite-dimensional graded vector space $V$ over $\C$ and let $\Aut(V)$ be the group of
grading preserving automorphisms of $V$.
Then, the multisegments with supercuspidal support determined by the graded dimension of $V$,
parameterize the finitely many $\Aut(V)$-orbits on the space of
linear transformations $A:V\rightarrow V$ that are homogeneous of degree $1$ (or alternatively, $-1$).
They also parameterize the irreducible components of the variety $\commvar(V)$ of pairs
of commuting linear transformation $A,B:V\rightarrow V$ of degree $1$ and $-1$ respectively.
Denote by $\comp_\m$ irreducible component corresponding to $\m$.
The conjecture of \cite{MR3866895}, inspired by \cite{MR2822235}*{Conjecture 18.1},
is that $Z(\m)$ is $\square$-irreducible if and only if $\comp_\m$ admits an open $\Aut(V)$-orbit.
(See \S\ref{sec: GLS} for more details.)
This condition can be explicated fairly concretely and can be efficiently
checked, at least probabilistically -- see \cite{MR3866895} where also a special case of this conjecture was proved,
and was related to singularities of Schubert varieties of type $A$.

\subsection{}
The main goal of this paper is to propose and substantiate a conjectural geometric criterion 
for the irreducibility of $Z(\m)\times Z(\m')$, provided that at least one of $Z(\m)$ and $Z(\m')$ is $\square$-irreducible.
More precisely, suppose that $\comp_{\m}\subset\commvar(V)$ and $\comp_{\m'}\subset\commvar(V')$.
Then, $\comp_{\m+\m'}\subset\commvar(V\oplus V')$ and we can embed $\comp_{\m}\times\comp_{\m'}$ diagonally in $\comp_{\m+\m'}$.
\begin{conjecture} \label{conj: intro1}
Suppose that at least one of $Z(\m)$ and $Z(\m')$ is $\square$-irreducible. Then
\[
Z(\m)\times Z(\m')\text{ is irreducible if and only if }
\Aut(V\oplus V')\cdot(\comp_{\m}\times\comp_{\m'})\text{ is dense in }\comp_{\m+\m'}.
\]
\end{conjecture}

\subsection{}
More generally, under the above assumption we give a conjectural geometric criterion for \eqref{eq: LI}.
(See \S\ref{sec: newconj} for more details.)
Recall that $\comp_\m$ comes with a distinguished $\Aut(V)$-invariant Zariski open subset $\comp_\m^\circ$
(containing the interior of $\comp_\m$ in $\commvar(V)$, but not equal to it in general).
Let $\Specvar$ be the subvariety of $\comp_{\m+\m'}$ consisting of the pairs $(A,B)$ satisfying the following
two conditions.
\begin{enumerate}
\item $A(V'),B(V')\subset V'$, and $(A\rest_{V'},B\rest_{V'})\in\comp_{\m'}^{\circ}$.
\item The induced pair on the quotient $V$ belongs to $\comp_{\m}^{\circ}$.
\end{enumerate}

\begin{conjecture} \label{conj: intro2}
Suppose that at least one of $Z(\m)$ and $Z(\m')$ is $\square$-irreducible. Then
\[
Z(\m+\m')\hookrightarrow Z(\m)\times Z(\m')\text{ if and only if }
\Aut(V\oplus V')\cdot\Specvar\text{ is dense in }\comp_{\m+\m'}.
\]
\end{conjecture}

This conjecture implies Conjecture \ref{conj: intro1}.

The geometric conditions in the two conjectures are in the spirit of the Geiss--Leclerc--Schr\"oer condition.
They are equally concrete and can be checked (at least probabilistically) very efficiently on a computer.

\subsection{}
We will give supportive evidence for Conjecture \ref{conj: intro2}, ergo, Conjecture \ref{conj: intro1}.
(See \S\ref{sec: mainev} for more details.)
For instance, we show that they hold whenever $Z(\m)$ is unitarizable (without restriction on $\m'$)
or if $\m$ is a ladder multisegment in the sense of \cite{MR3163355}.
For the latter, we use the results of \cite{MR3573961}.
In fact, the conjecture was forged as an attempt to explain the results of \cite{MR3573961} geometrically.
We also show that Conjecture \ref{conj: intro2} satisfies a number of non-trivial consistency checks.
While in all likelihood new ideas will be needed to establish the conjecture in general, we believe that
the attestation that we already have so far cannot be coincidental.

In the best-case scenario, Conjecture \ref{conj: intro2} may in fact hold without restriction on $\m$ and $\m'$.
However, we feel that at this stage it would be too parlous to postulate such a strong form of Conjecture \ref{conj: intro2}
since our evidence for this generality is indirect and far from conclusive.
The main reason is that we do not have a practical way to check the condition \eqref{eq: LI} independently in general.

In contrast, there are counterexamples for Conjecture \ref{conj: intro1} if we lift the assumptions on $\m$ and $\m'$.
In order to obtain a precise irreducibility criterion, we would need to characterize the condition
\[
Z(\m+\m')=\soc(Z(\m)\times Z(\m'))
\]
(which in general is stronger than \eqref{eq: LI})
and at present we do not have a conjectural geometric criterion for this.

In a different direction, a natural follow-up question, which we hope to study in the future, is to obtain a geometric insight
on $\soc(Z(\m)\times Z(\m'))$, assuming as before that at least one of $Z(\m)$ and $Z(\m')$ is $\square$-irreducible.

Finally, we point out that our conjectures do not seem to lie in the scope of the Langlands program.

The contents of the paper are as follows.
After introducing the relevant notation and the Zelevinsky classification (\S\ref{sec: notation})
we recall the notion of $\square$-irreducible representations and basic facts
about irreducibility of parabolic induction (\S\ref{sec: LI}).
In \S\ref{sec: GLS} we recall the geometric condition of Geiss--Leclerc--Schr\"oer and the conjecture
relating it to $\square$-irreducibility.
The heart of the paper is \S\ref{sec: newconj} where we state Conjectures \ref{conj: intro1} and \ref{conj: intro2}
and analyze the pertinent geometric conditions.
In \S\ref{sec: mainev} we state results confirming these conjectures in special cases and provide
several consistency checks for them.
These results are proved in \S\ref{sec: proofs} using Jacquet module techniques and other combinatorial tools
which are recalled in \S\ref{sec: prep}.

\subsection{Acknowledgment}
We would like to thank David Kazhdan for useful discussions.

The first-named author would like to thank the University of Vienna and the Institute for Mathematical Sciences,
National University of Singapore for their hospitality and support.

Both authors would like thank the University Institute of Mathematics Research of the University of Seville
for its hospitality and support.

\section{Notation and preliminaries} \label{sec: notation}
\subsection{}
Throughout the paper we fix a non-archimedean local field $F$ with normalized absolute value $\abs{\cdot}$.
For a non-negative integer $n$ let $\Reps_n$ denote the $\C$-linear, locally finite,
abelian category of complex, smooth, finite length (hence admissible)
representations of the group $\GL_n(F)$. Set
\[
\Reps=\oplus_{n\ge0}\Reps_n
\]
and denote the simple objects of $\Reps$ by
\[
\Irr=\coprod_{n\ge0}\Irr\Reps_n.
\]
In particular, write $\Irr\Reps_0=\{\one\}$.
We denote by
\[
\Cusp=\coprod_{n>0}\Cusp\Reps_n
\]
the subset of irreducible supercuspidal representations.
By abuse of notation we often write $\pi\in\Reps$ to mean that $\pi$ is an object of $\Reps$.

For $\tau,\pi\in\Reps$ we write $\tau\hookrightarrow\pi$ (resp., $\tau\twoheadrightarrow\pi$) if there exists an injective
(resp., surjective) morphism from $\tau$ to $\pi$.
If $\tau\in\Irr$ and $\pi\in\Reps$, we will write $\tau\le\pi$ for the condition that
$\tau$ occurs as a subquotient of $\pi$ (i.e., $\tau$ occurs in the Jordan--H\"older sequence $\JH(\pi)$ of $\pi$).
If $\tau$ occurs with multiplicity one in $\JH(\pi)$, then we will write $\tau\leunq\pi$.

As customary, normalized parabolic induction with respect to standard (block upper triangular) parabolic subgroup
will be denoted by $\times$. This is a bilinear biexact bifunctor
with associativity constraints given by induction in stages. In other words, $\times$ endows
$\Reps$ with the structure of a ring category with unit element $\one$.
The Grothendieck group
\[
\Gr=\oplus_{n\ge0}\Gr_n
\]
of $\Reps$ inherits a structure of a graded commutative ring.

For any $\pi\in\Reps_n$ and a character $\chi$ of $F^*$, we denote by $\pi\cdot\chi$ the representation obtained
from $\pi$ by twisting by the character $\chi\circ\det$. In particular, we write
\[
\rshft\pi=\pi\cdot\abs{\cdot},\ \ \lshft\pi=\pi\cdot\abs{\cdot}^{-1}.
\]
We denote by $\pi^\vee$ the contragredient of $\pi$ and by $\soc(\pi)$ the socle of $\pi$, i.e.,
the largest semisimple subobject of $\pi$. (If $\pi\ne0$, then $\soc(\pi)\ne0$.)
\begin{definition}
We say that $\pi\in\Reps$ is \emph{socle irreducible} (\SI) if $\soc(\pi)$ is irreducible and $\soc(\pi)\leunq\pi$.
\end{definition}

By the argument of \cite{MR863522}*{p. 173}, for any $\pi_1,\pi_2,\tau\in\Irr$ we have
\begin{equation} \label{eq: switch12}
\tau\hookrightarrow\pi_1\times\pi_2\iff\pi_2\times\pi_1\twoheadrightarrow\tau.
\end{equation}

For any set $A$ we denote by $\MS{A}$ the free commutative monoid generated by the elements of $A$.
It consists of finite formal sums of elements of $A$. The standard order on $\MS{A}$ will be denoted by $\le$.

We may view $\JH$ as a map from the objects of $\Reps$ to $\MS{\Irr}$.
Thus, for $\tau\in\Irr$ and $\pi\in\Reps$ we have $\rho\le\pi$ if and only if $\rho\le\JH(\pi)$ in $\MS{\Irr}$.

For any $\pi\in\Irr$ there exist $\rho_1,\dots,\rho_k\in\Cusp$ such that $\pi\le\rho_1\times\dots\times\rho_k$.
The supercuspidal support map
\[
\cs:\Irr\rightarrow\MS{\Cusp},\ \ \pi\mapsto\rho_1+\dots+\rho_k
\]
is well-defined and finite-to-one \cite{MR0579172}.

\subsection{Zelevinsky classification \cite{MR584084}}
A \emph{segment} is a nonempty finite set of the form
\[
\Delta=\{\rho_1,\dots,\rho_k\}
\]
where $\rho_i\in\Cusp$ and $\rho_{i+1}=\rshft{\rho}_i$ for all $i=1,\dots,k-1$. For any such $\Delta$ we set
\begin{gather*}
Z(\Delta):=\soc(\rho_1\times\dots\times\rho_k)\in\Irr,\ \
L(\Delta):=\soc(\rho_k\times\dots\times\rho_1)\in\Irr,\\
\supp(\Delta)=\rho_1+\dots+\rho_k\in\MS{\Cusp},\\
\Delta^\vee=\{\rho_k^\vee,\dots,\rho_1^\vee\},\ \ b(\Delta)=\rho_1, \ \ e(\Delta)=\rho_k,\\
\rshft\Delta=\{\rshft\rho_1,\dots,\rshft\rho_k\},\ \Delta^-=\{\rho_1,\dots,\rho_{k-1}\},\ \Delta^+=\{\rho_1,\dots,\rho_k,\rshft\rho_k\},\\
\lshft\Delta=\{\lshft\rho_1,\dots,\lshft\rho_k\},\ ^-\Delta=\{\rho_2,\dots,\rho_k\},\ ^+\Delta=\{\lshft\rho_1,\rho_1,\dots,\rho_k\}.
\end{gather*}
For compatibility, we also set $Z(\emptyset)=L(\emptyset)=\one$.
Let $\Seg$ be the set of all segments.
(Note that $^-\Delta,\Delta^-\in\Seg$ if and only if $k>1$; otherwise $^-\Delta=\Delta^-=\emptyset$.)
For $\Delta, \Delta'\in\Seg$ we write $\Delta\prec\Delta'$ (and say that $\Delta$ precedes $\Delta'$)
if $b(\Delta)\notin\Delta'$, $b(\Delta')\in\rshft\Delta$ and
$e(\Delta')\notin\Delta$. Thus,
\begin{equation} \label{eq: sprec}
\lshft{\Delta}\prec\Delta'\iff\Delta\prec\rshft{\Delta'}\iff b(\Delta')\in\Delta\text{ and }e(\Delta)\in\Delta'.
\end{equation}
We say that $\Delta,\Delta'\in\Seg$ are linked if either $\Delta\prec\Delta'$ or $\Delta'\prec\Delta$.
(Equivalently, $\Delta\cup\Delta'\in\Seg$ and $\Delta,\Delta'\subsetneq\Delta\cup\Delta'$.)

By definition, a \emph{multisegment} is an element of $\MS{\Seg}$.
We extend the maps $\Delta\mapsto\supp(\Delta)$ and $\Delta\mapsto\Delta^\vee$ to an additive map
\[
\supp:\MS{\Seg}\rightarrow\MS{\Cusp}
\]
and an additive involution $\m\mapsto\m^\vee$ on $\MS{\Seg}$, respectively.
Given $\m=\Delta_1+\dots+\Delta_k\in\MS{\Seg}$
we may enumerate the $\Delta_i$'s (in possibly more than one way) such that $\Delta_i\not\prec\Delta_j$ whenever $i<j$.
Then, the representation $\std{\m}=Z(\Delta_1)\times\dots\times Z(\Delta_k)$
is \SI\ and up to equivalence, depends only on $\m$. The main result of Zelevinsky is that the map
\[
\m\mapsto Z(\m):=\soc(\std{\m})
\]
is a bijection between $\MS{\Seg}$ and $\Irr$.
There is also a dual bijection given by $\m\mapsto L(\m):=\soc(\cstd{\m})$ where
$\cstd{\m}=L(\Delta_k)\times\dots\times L(\Delta_1)$ (under the same assumption on the order of the $\Delta_i$'s).
The latter is essentially the Langlands classification in this context.
We refer the reader to \cite{MR3573961}*{\S3} for a summary of the basic properties of these bijections.
In particular, $Z(0)=L(0)=\one$,
\begin{equation} \label{eq: contra}
Z(\m)^\vee=Z(\m^\vee)\text{ and }L(\m)^\vee=L(\m^\vee)\text{ for any $\m\in\MS{\Seg}$,}
\end{equation}
$\cs(Z(\m))=\cs(L(\m))=\supp\m$ and
\begin{equation} \label{eq: mult1}
Z(\m+\m')\leunq Z(\m)\times Z(\m')\text{ for any }\m,\m'\in\MS{\Seg}.
\end{equation}

The two bijections are related by Deligne--Lusztig type duality.
Let $^t$ denote the duality functor on $\Reps$ defined by Schneider--Stuhler in \cite{MR1471867},
composed with the contragredient (in order to make it covariant).
(See also \cite{MR3769724} for a more recent approach.)
Then
\begin{equation} \label{eq: ZtL}
Z(\m)^t=L(\m)
\end{equation}
for all $\m\in\MS{\Seg}$ and
\begin{equation} \label{eq: proddual}
(\pi\times\pi')^t=\pi'^t\times\pi^t\text{ for any }\pi,\pi'\in\Reps.
\end{equation}

Finally, let
\begin{equation} \label{def: MW}
^{\#}:\MS{\Seg}\rightarrow\MS{\Seg}
\end{equation}
be the bijection (in fact, involution) such that $L(\m)=Z(\m^\#)$.
Note that $^{\#}$ is not additive.

A combinatorial description of $\m^{\#}$ was given by M\oe glin--Waldspurger \cite{MR863522}.
We will recall it in \S\ref{sec: Zeleinv} below. (A different one was later on given in \cite{MR1371654}.)

\section{The conditions \texorpdfstring{$\LI(\m,\m')$}{bb} and \texorpdfstring{$\SLI(\m,\m')$}{cc};
\texorpdfstring{$\square$}{dd}-irreducibility} \label{sec: LI}
\subsection{}
Let $\m,\m'\in\MS{\Seg}$. The following properties will be our main concern.
\begin{definition}
\begin{enumerate}
\item We denote by $\LI(\m,\m')$ the following equivalent conditions (by \eqref{eq: switch12}, \eqref{eq: ZtL} and \eqref{eq: proddual}).
\begin{enumerate}
\item $Z(\m+\m')\hookrightarrow Z(\m)\times Z(\m')$.
\item $L(\m+\m')\hookrightarrow L(\m')\times L(\m)$.
\item $Z(\m')\times Z(\m)\twoheadrightarrow Z(\m+\m')$.
\item $L(\m)\times L(\m')\twoheadrightarrow L(\m+\m')$.
\end{enumerate}
\item We denote by $\SLI(\m,\m')$ the following equivalent conditions (by \eqref{eq: mult1} and \cite{MR3573961}*{Lemma 4.2}).
\begin{enumerate}
\item $Z(\m+\m')\simeq\soc(Z(\m)\times Z(\m'))$.
\item $\forall\tau\in\Irr$, $\tau\hookrightarrow Z(\m)\times Z(\m')\implies\tau=Z(\m+\m')$.
\item $Z(\m)\times Z(\m')\hookrightarrow\std{\m+\m'}$.
\item $L(\m+\m')\simeq\soc(L(\m')\times L(\m))$.
\item $\forall\tau\in\Irr$, $\tau\hookrightarrow L(\m')\times L(\m)\implies\tau=L(\m+\m')$.
\item $\cstd{\m+\m'}\twoheadrightarrow L(\m)\times L(\m')$.
\end{enumerate}
\end{enumerate}
(The notation $\LI(\m,\m')$ and $\SLI(\m,\m')$ stand for (strongly) ``subrepresentation+additive''.)
\end{definition}

Clearly, $\SLI(\m,\m')\implies\LI(\m,\m')$, while
\begin{equation} \label{eq: LIcont}
\text{the conditions $\LI(\m,\m')$ and $\LI(\m'^\vee,\m^\vee)$ are equivalent.}
\end{equation}

A simple sufficient condition for $\SLI(\m,\m')$ is the following (see e.g., \cite{MR3573961}*{Proposition 3.5}).
\begin{equation} \label{eq: LIsimp}
\text{$\SLI(\m,\m')$ is satisfied if for every segment $\Delta$ in $\m$ and $\Delta'$ in $\m'$ we have $\Delta\not\prec\Delta'$.}
\end{equation}

The conditions $\LI(\m,\m')$ and $\SLI(\m,\m')$ are instrumental for proving irreducibility.
More precisely, again by \eqref{eq: mult1},
\begin{subequations}
\begin{gather}
\text{$Z(\m+\m')$ is a direct summand of $Z(\m)\times Z(\m')\iff\LI(\m,\m')$ and $\LI(\m',\m)$.}\\
\begin{split}
\label{eq: irredcrit}
Z(\m)\times Z(\m')\text{ is irreducible }\iff Z(\m+\m')\simeq Z(\m)\times Z(\m')\\
\iff \SLI(\m,\m')\text{ and }\SLI(\m',\m)\iff\LI(\m,\m')\text{ and }\SLI(\m',\m).
\end{split}
\end{gather}
\end{subequations}

As we will recall below, it is possible that $Z(\m)\times Z(\m')$ is semisimple but not irreducible.
In particular, $\SLI(\m,\m')$ is strictly stronger than $\LI(\m,\m')$.

\subsection{$\square$-irreducible representations}

For any $\pi_i\in\Reps_{n_i}$, $i=1,2$ consider the standard (unnormalized) intertwining operators
\[
\pi_1\cdot\abs{\cdot}^s\times\pi_2\rightarrow\pi_2\times\pi_1\cdot\abs{\cdot}^s.
\]
They are given by convergent integrals for $\Re s\gg0$ and admit a meromorphic continuation (e.g.,
\cite{MR1989693}*{Ch. IV}).\footnote{The
intertwining operators depend implicitly on a choice of a Haar measure on the space of $n_1\times n_2$ matrices,
so technically they are only defined up to a positive scalar. However, for our purposes this is immaterial.}
The leading term in the Laurent expansion at $s=0$ is a non-zero intertwining operator
\[
R_{\pi_1,\pi_2}:\pi_1\times\pi_2\rightarrow\pi_2\times\pi_1.
\]

The following key definition is inspired by the work of Hernandez--Leclerc and Kang--Kashiwara--Kim--Oh.
\begin{definition} (See \cites{MR3314831, MR3758148, MR3866895})
An object $\pi$ of $\Reps$ (necessarily in $\Irr$) is called $\square$-irreducible if the following
equivalent conditions are satisfied.
\begin{enumerate}
\item $\pi\times\pi$ is irreducible.
\item $\End_{\Reps}(\pi\times\pi)=\C$.
\item $R_{\pi,\pi}$ is a scalar.
\item $\pi\times\pi$ is \SI.
\item $\pi=Z(\m)$ and $\SLI(\m,\m)$.
\end{enumerate}
\end{definition}

We denote by $\IrrS\subset\Irr$ the class of $\square$-irreducible representations.
Clearly, this class is invariant under contragredient and the duality $^t$.
The first example of an irreducible representation $\pi$ (of $\GL_8(F)$) which is not $\square$-irreducible
was given by Leclerc \cite{MR1959765}. Explicitly, if
\begin{equation} \label{eq: Lex}
\pi=Z(\m)\text{ where }\m=[1,2]+[-1,1]+[0,0]+[-2,-1],
\end{equation}
and for simplicity we write $[a,b]$ for the segment $\{\abs{\cdot}^a,\dots,\abs{\cdot}^b\}$
consisting of characters of $\GL_1(F)=F^*$, then
\begin{multline*}
\pi\times\pi=Z(\m+\m)\oplus Z([1,2]+[0,1]+[-1,2]+[-1,0]+[-2,1]+[-2,-1])\\
=Z(\m+\m)\oplus (Z([-1,2]+[-2,1])\times Z([1,2]+[0,1]+[-1,0]+[-2,-1])).
\end{multline*}
Note that $\pi^\vee=\pi^t=\pi$.


\subsection{}
An important property of $\square$-irreducible representations is the following.
\begin{proposition} \label{prop: sqrirred}  (\cite{MR3866895} which is based on \cites{MR3314831, MR3758148})
Suppose that $\pi$ is $\square$-irreducible.
Then, for any $\pi'\in\Irr$
\begin{enumerate}
\item $\pi\times\pi'$ and $\pi'\times\pi$ are \SI.
\item $\soc(\pi\times\pi')=\Im R_{\pi',\pi}$.
\item $\soc(\pi'\times\pi)=\Im R_{\pi,\pi'}$.
\item The following conditions are equivalent:
\begin{enumerate}
\item $\pi\times\pi'$ is irreducible.
\item $R_{\pi,\pi'}$ is an isomorphism.
\item $R_{\pi',\pi}$ is an isomorphism.
\item $\pi\times\pi'\simeq\pi'\times\pi$.
\item $\soc(\pi\times\pi')\simeq\soc(\pi'\times\pi)$.
\end{enumerate}
\end{enumerate}
\end{proposition}

\label{sec: ALOref}

\begin{corollary} \label{cor: irredcond}
Suppose that
\begin{equation} \label{eq: ALO}
\tag{ALO} \text{at least one of $Z(\m)$ or $Z(\m')$ is $\square$-irreducible}.
\end{equation}
Then, the following conditions are equivalent.
\begin{enumerate}
\item $\LI(\m,\m')$.
\item $\SLI(\m,\m')$.
\item The image of $R_{Z(\m'),Z(\m)}$ is isomorphic to $Z(\m+\m')$.
\end{enumerate}
In particular, $Z(\m)\times Z(\m')$ is irreducible $\iff$ both $\LI(\m,\m')$ and $\LI(\m',\m)$ are satisfied.
\end{corollary}

We do not know whether in general the condition $\LI(\m,\m')$ implies that
$Z(\m+\m')$ is a subrepresentation of the image of $R_{Z(\m'),Z(\m)}$.

\begin{remark} \label{rem: mult>1}
Let $\pi=Z(\m)\in\Irr$ and $\Pi=\pi\times\pi$.
\begin{enumerate}
\item By \eqref{eq: switch12}, every irreducible subrepresentation of $\Pi$ also occurs as a quotient of $\Pi$.
Hence, any irreducible subrepresentation $\tau$ of $\Pi$ such that $\tau\leunq\Pi$ is a direct summand of $\Pi$.
It easily follows that if $\Pi$ is multiplicity free (i.e., if every element of $\JH(\Pi)$ occurs with multiplicity one),
then $\Pi$ is semisimple, and in particular $\LI(\m,\m)$ is satisfied.

\item In Leclerc's example \eqref{eq: Lex}, $\Pi$ is of length two, and in particular multiplicity free.
Hence, in general the condition $\LI(\m,\m)$ is strictly weaker than $\SLI(\m,\m)$
(i.e., the $\square$-irreducibility of $\pi$).

\item The first-named author computed $\JH(\Pi)$ for all multisegments $\m$ consisting of at most 6 segments.
(The computation involves Kazhdan--Lusztig polynomials of the symmetric group $S_{12}$.
See \cite{Elmaunote} for more details.)
It turns out that in these cases, $\Pi$ is multiplicity free if and only if $\Pi$ has length one or two.
An example where this condition fails is
\[
\m=[1,3]+[-2,2]+[-1,1]+[0,0]+[-3,-1]
\]
in which $\Pi$ has length 9 with
\begin{multline*}
Z([-2,3]+[-3,2]+[1,3]+[-1,2]+[-2,1]+[0,1]+[-1,0]+[-3,-1])\\=
Z([-2,3]+[-3,2])\times Z([1,3]+[-1,2]+[-2,1]+[0,1]+[-1,0]+[-3,-1])
\end{multline*}
occurring with multiplicity 2 in $\JH(\Pi)$. Another example is
\[
\m=[2,4]+[-2,3]+[-1,2]+[0,1]+[-4,0]+[-3,-1]
\]
for which $\Pi$ has length 257 (the largest possible for $\m$ up to 6 segments), with
\begin{multline*}
Z([-2,4]+[-3,3]+[-4,2]+[2,4]+[-4,3]+[-1,2]+[-2,1]+[-3,0]+[0,0]+[-1,-1])=
\\Z([-2,4]+[-3,3]+[-4,2])\times Z([2,4]+[-4,3]+[-1,2]+[-2,1]+[-3,0])\times Z([0,0]+[-1,-1])
\end{multline*}
occurring with multiplicity 9 (again, the highest possible) in $\JH(\Pi)$.

\item Unfortunately, in contrast to the multiplicity freeness of $\Pi$, we do not have a practical way to
completely determine whether $\Pi$ is semisimple or whether the condition $\LI(\m,\m)$ is satisfied.
In fact, at the moment we are unable to refute the condition $\LI(\m,\m)$, or even the semisimplicity of $\Pi$,
in any single example.
\end{enumerate}
\end{remark}

Finally, we mention another simple property, which is a powerful tool to show $\square$-irreducibility (\cite{MR3866895}*{Lemma 2.10}).
\begin{equation} \label{eq: factsirrs}
\text{Suppose that $\pi_1,\pi_2\in\IrrS$, $\pi\hookrightarrow\pi_1\times\pi_2$ and $\pi\times\pi_1$ is irreducible.
Then, $\pi\in\IrrS$.}
\end{equation}

An interesting question, which will not be discussed here, is whether conversely, given a non-supercuspidal $\pi\in\IrrS$,
do there always exist $\one\ne\pi_1,\pi_2\in\IrrS$ such that
$\pi\hookrightarrow\pi_1\times\pi_2$ and $\pi\times\pi_1$ is irreducible?

\section{A geometric condition of Geiss--Leclerc--Schr\"oer} \label{sec: GLS}

\subsection{}
We recall Zelevinsky's geometric picture of his classification \cites{MR617466, MR783619, MR863522, MR1371654}.
Consider a finite-dimensional $\Cusp$-graded $\C$-vector space
\[
\Vect=\oplus_{\rho\in\Cusp}\Vect_\rho.
\]
Up to isomorphism, $\Vect$ is determined by its graded dimension
\[
\grdim\Vect=\sum_{\rho\in\Cusp}(\dim\Vect_\rho)\rho\in\MS{\Cusp}.
\]
Let
\[
\Aut(\Vect)=\prod_{\rho\in\Cusp}\GL(\Vect_\rho)
\]
be the group of grading preserving linear automorphisms of $\Vect$ and let
\[
\rshft \grend(\Vect)\text{ (resp., $\lshft \grend(\Vect)$)}
\]
be the vector space of (nilpotent) linear transformations $A:\Vect\rightarrow\Vect$
such that $A(\Vect_\rho)\subseteq\Vect_{\rshft\rho}$ (resp., $A(\Vect_\rho)\subseteq\Vect_{\lshft\rho}$) for all $\rho\in\Cusp$.
We will use the notational convention $\rlshft \grend(\Vect)$ to denote either of the spaces $\rshft \grend(\Vect)$ and $\lshft \grend(\Vect)$.
A similar praxis will apply in other instances.

The group $\Aut(\Vect)$ acts on each of the spaces $\rlshft \grend(\Vect)$ (by conjugation) with finitely many orbits
and the spaces $\rlshft \grend(\Vect)$ are in duality with respect to the $\Aut(\Vect)$-invariant pairing $(A,B)\mapsto\tr AB=\tr BA$.
The orbit of an element $A\in\rlshft \grend(\Vect)$ is determined by the non-negative integers $\rk A^i\rest_{\Vect_\rho}$, $i\ge0$, $\rho\in\Cusp$
(which are non-zero for only finitely many $i$'s and $\rho$'s).
The orbits are parameterized by the multisegments $\m$ such that $\supp(\m)=\grdim\Vect$.
Concretely, given such $\m=\sum_{i\in I}\Delta_i\in\MS{\Seg}$, \label{sec: OOm}
an element $A\in\rlshft \grend(\Vect)$ belongs to the orbit $\rlshft\OO_{\m}$ corresponding to $\m$ if and only if
there exists a graded basis $\xbasis_{\rho,i}$, $i\in I$, $\rho\in\Delta_i$ (graded by $\rho$) of $\Vect$ for which
$A$ has a graded Jordan form
\begin{equation} \label{eq: defA}
A\xbasis_{\rho,i}=\xbasis_{\rlshft\rho,i}
\end{equation}
where we set $\xbasis_{\rho,j}=0$ if $\rho\notin\Delta_j$.

The $\Aut(\Vect)$-orbits in $\rlshft \grend(\Vect)$ are also in one-to-one correspondence with the irreducible components of the commuting variety
\[
\commvar(\Vect)=\{(A,B)\in \rshft \grend(\Vect)\times \lshft \grend(\Vect):AB=BA\}
\]
\cite{MR0390138}. These bijections take an orbit $\OO$ in $\rlshft \grend(\Vect)$ to the Zariski closure
of $\rlshft p_{\Vect}^{-1}(\OO)$ where \label{sec: p_V}
\[
\rlshft p_{\Vect}:\commvar(\Vect)\rightarrow \rlshft \grend(\Vect)
\]
are the canonical projections.
Let $\rlshft\comp_{\m}$ be the irreducible components of $\commvar(\Vect)$ corresponding to $\rlshft\OO_{\m}$.
They are related by $\rshft\comp_{\m}=\lshft\comp_{\m^\#}$ \cite{MR863522} where we recall that $\m^\#$
was defined in \eqref{def: MW}.
We also write
\[
\rlshft\comp^\circ_{\m}=\rlshft p_{\Vect}^{-1}(\rlshft\OO_{\m}),
\]
so that $\rlshft\comp_{\m}$ is the Zariski closure of $\rlshft\comp^\circ_{\m}$.
Note that $\rlshft\comp^\circ_{\m}$ is $\Aut(\Vect)$-invariant and open in $\rlshft\comp_{\m}$.
(We caution that $\rlshft\comp^\circ_{\m}$ contains the interior $\rlshft\comp_{\m}\setminus
\cup_{\m'\ne\m}\rlshft\comp_{\m'}$ of $\rlshft\comp_\m$ in $\commvar(\Vect)$, but the inclusion is strict in general.)

\subsection{}
In general, there are infinitely many $\Aut(\Vect)$-orbits in $\commvar(\Vect)$.
Following Geiss--Leclerc--Schr\"oer \cite{MR2822235} we make the following definition.
(As before, $\m\in\MS{\Seg}$ and $\supp\m=\grdim\Vect$.)

\label{sec: GLSdef}
\begin{definition}
We say that the condition $\GLS(\m)$ holds if $\rshft\comp_\m$ admits an open (i.e., dense) $\Aut(\Vect)$-orbit.
\end{definition}

By \cite{MR3866895}*{Remark 4.6}, the analogous condition for $\lshft\comp_\m$ is equivalent.
Therefore, from now on we will exclusively work with $\rshft\comp_\m$.
In order to simplify the notation, we will henceforth write $\comp_{\m}$ instead of $\rshft\comp_\m$.
(and similarly for $\OO_{\m}$ and $\comp_{\m}^{\circ}$).

Note that an open $\Aut(\Vect)$-orbit in $\comp_\m$, if exists, is necessarily contained in $\comp_\m^\circ$.

\begin{conjecture}(\cite{MR3866895}, following \cite{MR2822235}*{Conjecture 18.1}) \label{conj: GLS}
For any multisegment $\m$, $Z(\m)$ is $\square$-irreducible if and only if $\GLS(\m)$ holds.
\end{conjecture}


\subsection{} \label{sec: expGLS}
As in \cite{MR3866895}*{\S4}, it is advantageous to explicate the condition $\GLS(\m)$ by linearization.
More precisely, fix $\repn_{\m}\in\OO_\m$ and write it in the form \eqref{eq: defA} for a suitable basis
$\xbasis_{\rho,i}$, $\rho\in\Delta_i$, $i\in I$.
The map $\lshft p_{\Vect}$ identifies the fibre $\rshft p_{\Vect}^{-1}(\repn_{\m})$ with the centralizer $\cnt_\m$ of $\repn_{\m}$ in $\lshft \grend(\Vect)$. \label{sec: defC}
The stabilizer $\stb_\m$ of $\repn_{\m}$ in $\Aut(\Vect)$ acts linearly on $\cnt_{\m}$. (Note that $\stb_{\m}$ is usually not reductive.)
Thus, $\GLS(\m)$ holds if and only if $\stb_{\m}$ admits an open (i.e., dense) orbit in $\cnt_{\m}$.
Passing to the Lie algebra, we can rephrase it by saying that there exists $\lambda\in \cnt_{\m}$ such that $[\Lieg,\lambda]=\cnt_{\m}$
where $\Lieg_\m$ is the Lie algebra of $\stb_{\m}$.

The vector space $\cnt_{\m}$ was explicated in \cite{MR863522}*{Lemme II.4} as follows.
Let \label{sec: defXm} $X_\m$ be the set
\[
X_\m=\{(i,j)\in I\times I:\Delta_i\prec\Delta_j\}.
\]
Then, $\cnt_{\m}$ admits a basis $\alpha_{i,j}=\alpha_{i,j}^{\m}$, $(i,j)\in X_\m$ given by
\[
\alpha_{i,j}(\xbasis_{\rho,l})=\delta_{j,l}\xbasis_{\lshft{\rho},i},\ \ \rho\in\Delta_l, l\in I
\]
where $\delta_{r,s}$ is Kronecker's delta.
(Recall that by convention, $\xbasis_{\rho,j}=0$ if $\rho\notin\Delta_j$.)
For any $\lambda\in \cnt_{\m}$, we write the coordinates of $\lambda$ with respect to the basis $\alpha_{i,j}$ as
$\lambda_{i,j}\in\C$, $(i,j)\in X_\m$. Thus,
\[
\lambda(\xbasis_{\rho,j})=\sum_{i\in I:(i,j)\in X_\m}\lambda_{i,j}\xbasis_{\lshft\rho,i},\ \ \rho\in\Delta_j.
\]

Similarly, the group $\stb_{\m}$ and its Lie algebra $\Lieg_\m$ are described in \cite{MR863522}*{Lemme II.5}. Let
$Y_\m$ be the set
\[
Y_\m=\{(i,j)\in I\times I:\Delta_i\prec\rshft\Delta_j\}.
\]
Then, as a vector space, $\Lieg_\m$ has a basis $\beta_{i,j}=\beta_{i,j}^{\m}$, $(i,j)\in Y_\m$ given by
\[
\beta_{i,j}(\xbasis_{\rho,l})=\delta_{j,l}\xbasis_{\rho,i},\ \ \rho\in\Delta_l, l\in I.
\]
Thus, $g\in\Aut(\Vect)$ is in $\stb_{\m}$ if and only if
\[
g(\xbasis_{\rho,j})=\sum_{i:(i,j)\in Y_\m}g_{i,j}\xbasis_{\rho,i},\ \ \rho\in\Delta_j
\]
for scalars $g_{i,j}$, $(i,j)\in Y_\m$. We will call $g_{i,j}$ the coordinates of $g$.

Moreover, we have
\[
[\beta_{i,j},\alpha_{k,l}]=\delta_{j,k}\alpha_{i,l}-\delta_{i,l}\alpha_{k,j},\ \ (i,j)\in Y_\m,\ (k,l)\in X_\m
\]
where for convenience we set $\alpha_{r,s}=0$ if $(r,s)\notin X_\m$.

Therefore, by passing to the dual map we can explicate the surjectivity of the linear map
$[\cdot,\lambda]:\Lieg_\m\rightarrow \cnt_{\m}$ as follows.
\begin{lemma} \label{lem: itfvij}
Consider the $\C$-vector space $\C^{Y_\m}$ with standard basis $\{y_{i,j}:(i,j)\in Y_\m\}$.
Then, $\GLS(\m)$ is satisfied if and only if there exists $\lambda\in \cnt_{\m}$, with coordinates $\lambda_{i,j}\in\C$, $(i,j)\in X_\m$,
such that the vectors
\[
v_{i,j}^\m(\lambda):=\sum_{k\in I:(k,j)\in X_\m,(i,k)\in Y_\m}\lambda_{k,j}y_{i,k}-
\sum_{l\in I:(l,j)\in Y_\m,(i,l)\in X_\m}\lambda_{i,l}y_{l,j},\ \ (i,j)\in X_\m
\]
are linearly independent in $\C^{Y_\m}$.
\end{lemma}
Note that the linear independence condition of Lemma \ref{lem: itfvij} is open and $\stb_{\m}$-invariant in $\lambda\in \cnt_{\m}$.
This condition is also easy to check (at least probabilistically) on a computer.
(See \cite{MR3866895}*{Remark 4.9} for a hypothetical approach to make this deterministic.)

The condition $\GLS(\m)$ is invariant under the involutions $\m\mapsto\m^\vee$ and $\m\mapsto\m^\#$ \cite{MR3866895}*{Remark 4.18}.

The main result of \cite{MR3866895} is that Conjecture \ref{conj: GLS} holds in the special case
where $\m=\sum_{i\in I}\Delta_i$ is regular, i.e., when $e(\Delta_i)\ne e(\Delta_j)$ and $b(\Delta_i)\ne b(\Delta_j)$ for all $i\ne j$.
In this case, the condition is also related, somewhat surprisingly, to smoothness of Schubert varieties of type $A$.

It was also verified computationally that Conjecture \ref{conj: GLS} holds for $\m$ consisting of up to 6 segments.

\section{Conjectural geometric conditions for \texorpdfstring{$\LI(\m,\m')$}{aa}} \label{sec: newconj}

In this section, which is the heart of the paper, we introduce a new, and easy to check, geometric condition
$\LC(\m,\m')$ pertaining to two multisegments $\m,\m'$.
This condition is in the spirit of the condition $\GLS(\m)$ discussed in \S\ref{sec: GLSdef}.
We conjecture that $\LC(\m,\m')$ is equivalent to $\LI(\m,\m')$, at least
under the condition \eqref{eq: ALO} of \S\ref{sec: ALOref}.
In particular, in this case it would give a practical condition for irreducibility of parabolic induction.

\subsection{}

We continue to work with the geometric setup and the notation of the previous section. 

Let $\Vect$ and $\Vect'$ be two finite-dimensional $\Cusp$-graded vector spaces over $\C$.
Consider the direct sum $\tilde\Vect=\Vect\oplus\Vect'$ and the short exact sequence
\[
0\rightarrow\Vect\xrightarrow{\embd}\tilde\Vect\xrightarrow{\srjct}\Vect'\rightarrow0.
\]
Define the spaces $\rshft \grendpres(\tilde\Vect)$ and $\lshft \grendpres(\tilde\Vect)$ by
\[
\rlshft \grendpres(\tilde\Vect)=\{A\in\rlshft \grend(\tilde\Vect):\srjct\circ A\circ\embd=0\}=\{A\in \rlshft \grend(\tilde\Vect):A(\Vect)\subset\Vect\}
\]
and let
\[
\commvar^\embd(\tilde\Vect)=\commvar(\tilde\Vect)\cap (\rshft \grendpres(\tilde\Vect)\times \lshft \grendpres(\tilde\Vect))
=\{(A,B)\in\commvar(\tilde\Vect):A(\Vect),B(\Vect)\subset\Vect\}.
\]
The stabilizer $\ParAut$ of $\embd(\Vect)$ in $\Aut(\tilde\Vect)$
is a parabolic subgroup of $\Aut(\tilde\Vect)$ with Levi quotient isomorphic to $\Aut(\Vect)\times\Aut(\Vect')$
and unipotent radical
\[
\UnipAut=\{g\in\Aut(\tilde\Vect):g\circ\embd=\embd\text{ and }\srjct\circ g=\srjct\},
\]
which is commutative.

Any $A\in \rlshft \grendpres(\tilde\Vect)$ restricts to a map in $\rlshft \grend(\Vect)$ and induces
a map in $\rlshft \grend(\Vect')$. We denote by
\[
\rlshft p^{\embd}:\rlshft \grendpres(\tilde\Vect)\rightarrow\rlshft \grend(\Vect)\times \rlshft \grend(\Vect')
\]
the resulting $\ParAut$-equivariant surjective maps
(with $\UnipAut$ acting trivially on the target).
(These projections should not be confused with the projections $\rlshft p_{\Vect}$ defined in \S\ref{sec: p_V}.)
Let
\[
p_\embd=(\rshft p^{\embd},\lshft p^{\embd})\rest_{\commvar^{\embd}(\tilde\Vect)}:
\commvar^{\embd}(\tilde\Vect)\rightarrow\commvar(\Vect)\times\commvar(\Vect').
\]
Again, $p_\embd$ is $\ParAut$-equivariant (and in particular, $\UnipAut$-invariant) and surjective.

Define similarly $\rlshft \grendpress(\tilde\Vect)$, $\commvar^{\embd'}(\tilde\Vect)$, $\commvar^{\embd'}(\tilde\Vect)$, $\ParAutt$, $\UnipAutt$ and
\[
p_{\embd'}:\commvar^{\embd'}(\tilde\Vect)\rightarrow\commvar(\Vect')\times\commvar(\Vect)
\]
by interchanging the roles of $\Vect$ and $\Vect'$. The parabolic subgroups $\ParAut$ and $\ParAutt$ are opposite with
\[
\LeviAut:=\ParAut\cap\ParAutt=\{g\in\Aut(\tilde\Vect):g(\Vect)=\Vect\text{ and }g(\Vect')=\Vect'\}\simeq\Aut(\Vect)\times\Aut(\Vect').
\]
Finally, let
\[
\dembed:\commvar(\Vect)\times\commvar(\Vect')\rightarrow\commvar(\tilde\Vect)
\]
be the diagonal embedding. Its image is
\[
\commvar^{\embd}(\tilde\Vect)\cap\commvar^{\embd'}(\tilde\Vect)
=\{(A,B)\in\commvar(\tilde\Vect):A(\Vect),B(\Vect)\subset\Vect\text{ and }A(\Vect'),B(\Vect')\subset\Vect'\}
\]
and $p_{\embd}\circ\dembed=\id$.

\subsection{}  \label{sec: mm'}
Now let $\m,\m'$ be two multisegments such that $\supp\m=\grdim\Vect$ and $\supp\m'=\grdim\Vect'$,
and let $\n=\m+\m'$, so that $\supp\n=\grdim\tilde\Vect$. As in \S\ref{sec: OOm},
let $\OO_{\m}$ (resp., $\OO_{\m'}$, $\OO_{\n}$) be the $\Aut(\Vect)$ (resp., $\Aut(\Vect')$, $\Aut(\tilde\Vect)$)
orbit in $\rshft \grend(\Vect)$ (resp., $\rshft \grend(\Vect')$, $\rshft \grend(\tilde\Vect)$)
corresponding  to $\m$ (resp., $\m'$, $\n$).
Recall that $\comp^\circ_{\m}=\rshft p_{\Vect}^{-1}(\OO_{\m})$,
$\comp^\circ_{\m'}=\rshft p_{\Vect'}^{-1}(\OO_{\m'})$, $\comp^\circ_{\n}=\rshft p_{\tilde\Vect}^{-1}(\OO_{\n})$
and $\comp_{\m}$, $\comp_{\m'}$, $\comp_{\n}$ are the respective (Zariski) closures.

Let
\[
\Specvar_{\m,\m'}^{\embd}=\cl{\comp_{\n}\cap p_{\embd}^{-1}(\comp_{\m}^\circ\times\comp_{\m'}^\circ)}
\subset\comp_{\n}\cap p_{\embd}^{-1}(\comp_{\m}\times\comp_{\m'})
\]
and
\[
\Specvar_{\m',\m}^{\embd'}=\cl{\comp_{\n}\cap p_{\embd'}^{-1}(\comp_{\m'}^\circ\times\comp_{\m}^\circ)}
\subset\comp_{\n}\cap p_{\embd'}^{-1}(\comp_{\m'}\times\comp_{\m}).
\]
(We do not know whether the above inclusions can be strict, but fortunately this will not matter in what follows.)

Let
\[
\Prodcomp_{\m,\m'}=\Specvar_{\m,\m'}^{\embd}\cap\Specvar_{\m',\m}^{\embd'}.
\]
Note that
\[
\Prodcomp_{\m,\m'}=\dembed(\comp_{\m}\times\comp_{\m'})
\]
since $\supset$ clearly holds while
\[
\Prodcomp_{\m,\m'}\subset p_{\embd}^{-1}(\comp_{\m}\times\comp_{\m'})\cap
p_{\embd'}^{-1}(\comp_{\m'}\times\comp_{\m})=\dembed(\comp_{\m}\times\comp_{\m'}).
\]


Clearly, $\Specvar_{\m',\m}^{\embd'}$ is $\ParAutt$-invariant and hence,
$\cl{\Aut(\tilde\Vect)\cdot\Specvar_{\m',\m}^{\embd'}}=\cl{\UnipAut\cdot\Specvar_{\m',\m}^{\embd'}}$,
since $\UnipAut\ParAutt$ is dense in $\Aut(\tilde\Vect)$.

Also, $\Prodcomp_{\m,\m'}$ is $\LeviAut$-invariant and hence, $\ParAut\cdot\Prodcomp_{\m,\m'}=\UnipAut\cdot\Prodcomp_{\m,\m'}$.

We can now formulate the main geometric conditions.

\begin{proposition} \label{prop: LCint}
The following conditions are equivalent.
\begin{subequations}
\begin{gather} \label{cond: GVorb}
\Aut(\tilde\Vect)\cdot\Specvar_{\m',\m}^{\embd'}\text{ (or equivalently, $\UnipAut\cdot\Specvar_{\m',\m}^{\embd'}$)
is dense in }\comp_{\n}. \\
\label{cond: parorb}
\ParAut\cdot\Prodcomp_{\m,\m'}\text{ (or equivalently, $\UnipAut\cdot\Prodcomp_{\m,\m'}$) is dense in }\Specvar_{\m,\m'}^{\embd}.
\end{gather}
\end{subequations}
\end{proposition}

\begin{definition}
We denote the above equivalent conditions by $\LC(\m,\m')$. (This stands for ``subrepresentation+geometric'')
We denote by $\IC(\m,\m')$ (for ``irreducible+geometric'') the condition
\[
\Aut(\tilde\Vect)\cdot\Prodcomp_{\m,\m'}\text{ is dense in }\comp_{\n}.
\]
\end{definition}

Note that under the proposition above, $\IC(\m,\m')$ is simply the conjunction of $\LC(\m,\m')$ and $\LC(\m',\m)$.
Indeed, $\IC(\m,\m')$ clearly implies \eqref{cond: GVorb} and its symmetric analog (interchanging $\m$ and $\m'$).
Conversely, \eqref{cond: parorb} and the symmetric counterpart of \eqref{cond: GVorb} imply that
\[
\cl{\Aut(\tilde\Vect)\cdot\Prodcomp_{\m,\m'}}=
\cl{\Aut(\tilde\Vect)\ParAut\cdot\Prodcomp_{\m,\m'}}=
\cl{\Aut(\tilde\Vect)\cdot\Specvar_{\m,\m'}^{\embd}}=\comp_{\n}.
\]

We will give some more equivalent conditions for $\LC(\m,\m')$ in \S\ref{sec: moreconds} below
and ultimately prove the proposition in \S\ref{sec: concprof}. \label{sec: ICdefeq}

\subsection{} \label{subsec: newconj}
We now state the new conjecture.

\begin{conjecture} \label{conj: weakform} 
Suppose \eqref{eq: ALO} of \S\ref{sec: ALOref}.
Then,
\begin{enumerate}
\item The conditions $\LI(\m,\m')$ and $\LC(\m,\m')$ are equivalent.
\item $Z(\m)\times Z(\m')$ is irreducible if and only if $\IC(\m,\m')$.
\end{enumerate}
\end{conjecture}

Note that the second part of the conjecture follows from the first one by Corollary \ref{cor: irredcond}
and the discussion above.

More generally, it would be desirable to have a geometric/combinatorial grasp on $\soc(Z(\m)\times Z(\m'))$
under \eqref{eq: ALO}. We hope to get back to this question in the future.


A more ambitious formulation of Conjecture \ref{conj: weakform} is the following.

\begin{question} \label{qu: strongform}
Are the conditions $\LI(\m,\m')$ and $\LC(\m,\m')$ equivalent even without assuming \eqref{eq: ALO} ?
\end{question}

At any rate, an affirmative answer to this question would not directly give an irreducibility
criterion for $\pi_1\times\pi_2$ when neither $\pi_1$ nor $\pi_2$ is $\square$-irreducible.

\begin{remark}
\begin{enumerate}
\item As we shall see in \S\ref{sec: GLSLC} below, the condition $\GLS(\m)$ implies $\LC(\m,\m)$ for any $\m$.
(This is consistent with Conjecture \ref{conj: GLS} and the fact that $\SLI(\m,\m)$ implies $\LI(\m,\m)$.)

\item By inspection, for $\m$ consisting of up to 6 segments, $\LC(\m,\m)$ is equivalent to the multiplicity freeness
of $Z(\m)\times Z(\m)$, and in particular it implies $\LI(\m,\m)$ in these cases (cf. Remark \ref{rem: mult>1}).
Leclerc's example \eqref{eq: Lex} shows that in general, the condition $\LC(\m,\m)$ is strictly weaker than $\GLS(\m)$.


\item If neither $\GLS(\m)$ nor $\GLS(\m')$ is satisfied, then at the moment we do not have a conjectural geometric criterion
for either $\SLI(\m,\m')$ or the irreducibility of $Z(\m)\times Z(\m')$ (cf. \eqref{eq: irredcrit}).
\end{enumerate}
\end{remark}



\subsection{} \label{sec: moreconds}
Before proving the equivalence of the conditions in Proposition \ref{prop: LCint} we will first linearize
and explicate them in a way similar to what was done for the condition $\GLS(\m)$ in \S\ref{sec: expGLS}.

Write $\m=\sum_{i\in I}\Delta_i$ and $\m'=\sum_{i'\in I'}\Delta_{i'}$ with $I\cap I'=\emptyset$.
Fix $\repn_{\m}\in\OO_{\m}$ and $\repn_{\m'}\in\OO_{\m'}$ and let $\repn_{\n}=\repn_{\m}\oplus \repn_{\m'}\in\OO_{\n}$.
We take a graded basis $\xbasis_{\rho,i}$, $\rho\in\Delta_i$, $i\in I$ (resp., $i\in I'$) for $\Vect$ (resp. $\Vect'$)
for which $\repn_{\m}$ (resp., $\repn_{\m'}$) has a graded Jordan normal form \eqref{eq: defA}.
Thus, $\repn_{\n}$ has a graded Jordan form with respect to the union $\xbasis_{\rho,i}$, $\rho\in\Delta_i$, $i\in I\cup I'$.
Let \label{def: Xmm'}
\[
X_{\m,\m'}=\{(i,i')\in I\times I':\Delta_i\prec\Delta_{i'}\}
\]
and
\[
Y_{\m,\m'}=\{(i,i')\in I\times I':\Delta_i\prec\rshft\Delta_{i'}\},
\]
so that $X_{\m}=X_{\m,\m}$ and $Y_\m=Y_{\m,\m}$ (see \S\ref{sec: defXm}).
We have $X_\n=X_{\m}\cup X_{\m'}\cup X_{\m,\m'}\cup X_{\m',\m}$ and $Y_\n=Y_{\m}\cup Y_{\m'}\cup Y_{\m,\m'}\cup Y_{\m',\m}$ (disjoint unions).
As in \S\ref{sec: defC}, let $\cnt_{\m}$, $\cnt_{\m'}$ and $\cnt_{\n}$ be the centralizers of $\repn_{\m}$ $\repn_{\m'}$ and $\repn_{\n}$ in
$\lshft \grend(\Vect)$, $\lshft \grend(\Vect')$ and $\lshft \grend(\tilde\Vect)$ respectively.
Similarly, let $\stb_{\m}$ (resp., $\stb_{\m'}$, $\stb_{\n}$) be the centralizer of $\repn_{\m}$ (resp., $\repn_{\m'}$, $\repn_{\n}$) in $\Aut(\Vect)$ (resp., $\Aut(\Vect')$, $\Aut(\tilde\Vect)$).
Let $\cntpres=\cnt_{\n}\cap\lshft \grendpres(\tilde\Vect)$.
Thus, $\cntpres$ consists of the elements of $\cnt_{\n}$ whose $X_{\m',\m}$-coordinates vanish.
Define similarly $\cntpress$ and let $\cntbw=\cntpres\cap\cntpress\simeq \cnt_{\m}\oplus \cnt_{\m'}$.
Let
\begin{gather*}
\unipstb=\stb_{\n}\cap\UnipAut=\{g\in\stb_{\n}:g_{i,j}=\delta_{i,j}\text{ for all }(i,j)\in Y_{\n}\setminus Y_{\m,\m'}\},\\
\parstb=\stb_{\n}\cap\ParAut=\{g\in\stb_{\n}:g_{i,j}=0\text{ for all }(i,j)\in Y_{\m',\m}\},\\
\unipstbb=\stb_{\n}\cap\UnipAutt=\{g\in\stb_{\n}:g_{i,j}=\delta_{i,j}\text{ for all }(i,j)\in Y_{\n}\setminus Y_{\m',\m}\},\\
\parstbb=\stb_{\n}\cap\ParAutt=\{g\in\stb_{\n}:g_{i,j}=0\text{ for all }(i,j)\in Y_{\m,\m'}\},\\
\levistb=\stb_{\n}\cap\LeviAut=\parstb\cap\parstbb\simeq \stb_{\m}\times \stb_{\m'},
\end{gather*}
so that
\[
\parstb=\levistb\ltimes\unipstb,\ \ \parstbb=\levistb\ltimes\unipstbb.
\]
The Lie algebra of $\unipstb$ is spanned by $\beta_{i,j}^{\n}$, $(i,j)\in Y_{\m,\m'}$.

\begin{proposition} \label{prop: LCint2}
We have $\cl{\stb_{\n}\cdot \cntpress}=\cl{\unipstb\cdot\cntpress}$
and $\parstb\cdot\cntbw=\unipstb\cdot\cntbw$.
Moreover, the following conditions are equivalent. 
\begin{enumerate}
\item \label{part: gc2}
$\stb_{\n}\cdot \cntpress$ (or equivalently, $\unipstb\cdot\cntpress$) is dense in $\cnt_{\n}$.
\item \label{part: trf2'}
$\parstb\cdot\cntbw$ (or equivalently. $\unipstb\cdot\cntbw$) is dense in $\cntpres$.
\item \label{part: linind} There exist $\lambda\in \cnt_{\m}$ and $\lambda'\in \cnt_{\m'}$
with coordinates $\lambda_{i,j}$, $(i,j)\in X_{\m}$
and $\lambda'_{i,j}$, $(i,j)\in X_{\m'}$ respectively, such that the vectors
$v_{i,j}^{\m,\m'}(\lambda,\lambda')$, $(i,j)\in X_{\m,\m'}$ given by
\begin{equation} \label{eq: vectors}
v_{i,j}^{\m,\m'}(\lambda,\lambda')=
\sum_{r\in I:(i,r)\in X_{\m},(r,j)\in Y_{\m,\m'}}\lambda_{i,r}y_{r,j}-
\sum_{s\in I':(s,j)\in X_{\m'},(i,s)\in Y_{\m,\m'}}\lambda'_{s,j}y_{i,s}
\end{equation}
are linearly independent in the complex vector space $\C^{Y_{\m,\m'}}$ with standard basis $\{y_{i,j}:(i,j)\in Y_{\m,\m'}\}$.
\end{enumerate}
Moreover, $\LC(\m,\m')$ is equivalent to the conditions above.
\end{proposition}

\label{sec: GLSLC}

We will prove the proposition below.
\begin{remark} \label{rem: LCsimcons}
Assume that the above proposition holds. Then,
\begin{enumerate}
\item By the same argument as in the end of \S\ref{sec: ICdefeq} (or alternatively, using Remark \ref{rem: chkclostrns} below),
$\IC(\m,\m')$ is equivalent to the condition
\[
\stb_{\n}\cdot\cntbw\text{ is dense in }\cnt_{\n}.
\]
\item For any $\m$, $\GLS(\m)$ implies $\LC(\m,\m)$, since
in the notation of Lemma \ref{lem: itfvij} we have
\[
v_{i,j}^\m(\lambda)=v_{i,j}^{\m,\m}(\lambda,\lambda)
\]
for $(i,j)\in X_\m$ and $\lambda\in \cnt_{\m}$.
\item The linear independence of \eqref{eq: vectors}
is an open and $\stb_{\m}\times \stb_{\m'}$-invariant condition in $(\lambda,\lambda')\in \cnt_{\m}\times \cnt_{\m'}$.
\item The conditions $\LC(\m,\m')$ and $\LC(\m'^\vee,\m^\vee)$ are equivalent,
since $X_{\m,\m'}=X_{\m'^\vee,\m^\vee}$ and $Y_{\m,\m'}=Y_{\m'^\vee,\m^\vee}$.
Thus, taking into account \eqref{eq: LIcont}, in Conjecture \ref{conj: weakform} we may assume without loss of generality that
$Z(\m)$ is $\square$-irreducible.
\item Condition \ref{part: linind} is the least conceptual but the most practical to check.
It can be easily implemented on a computer, at least as a probabilistic algorithm.
Nonetheless, it would be interesting to replace it by a deterministic criterion.
\end{enumerate}
\end{remark}

\subsection{Proof of first part of Proposition \ref{prop: LCint2}}

Since $\cntbw$ is $\levistb$-invariant and
$\parstb=\levistb\ltimes\unipstb$, it is clear that
$\parstb\cdot\cntbw=\unipstb\cdot\cntbw$.

Also, since $\unipstb\parstbb$ is dense in $\stb_{\n}$ (e.g., by considering the Lie algebras)
and $\cntpress$ is $\parstbb$-invariant, we have
$\cl{\stb_{\n}\cdot \cntpress}=\cl{\unipstb\cdot\cntpress}$.

Next, we show the equivalence of conditions \ref{part: gc2}, \ref{part: trf2'} and \ref{part: linind}
of Proposition \ref{prop: LCint2}.
We use the following simple general criterion.

\begin{remark} \label{rem: chkclostrns}
Suppose that $G$ is a linear algebraic group with a rational (linear) representation on a finite-dimensional vector space $W$
and let $W'$ be a subspace of $W$. Then,
\[
\text{$G\cdot W'$ is dense in $W$ if and only if there exists $w\in W'$ such that $\Lie(G)\cdot w+W'=W$.}
\]
\end{remark}

Returning to the case at hand, we identify
\[
\cntpres/\cnt_{\n}^{\embd,\embd}\simeq\cnt_{\n}/\cntpress
\]
with $\C^{X_{\m,\m'}}$ by considering the $X_{\m,\m'}$-coordinates.
For any $\mu\in\cntpress$ the linear map
\[
u\in\Lie(\unipstb)\mapsto u\cdot\mu+\cntpress
\]
gives rise to a linear map
\[
L_\mu:\C^{Y_{\m,\m'}}\rightarrow\C^{X_{\m,\m'}}
\]
which depends only on the image $(\lambda,\lambda')$ of $\mu$ under the projection $\cntpress\rightarrow \cnt_{\m}\oplus \cnt_{\m'}$.
Thus, by the remark above, the conditions \ref{part: gc2} and \ref{part: trf2'} of Proposition \ref{prop: LCint2}
are both equivalent to the surjectivity of $L_\mu$ for some $\mu\in\cntbw$.
Condition \ref{part: linind} is merely an explication of this (or more precisely, the injectivity of the dual map).

\subsection{}

It remain to show the equivalence of each of the conditions in Proposition \ref{prop: LCint}
with the corresponding condition in Proposition \ref{prop: LCint2}, thus completing
the proof of the two propositions.

We use the following simple result.

\begin{lemma} \label{lem: svtrns}
Let $G$ be a linear algebraic group acting algebraically on quasi-affine varieties $X$ and $Y$
and let $p:X\rightarrow Y$ be a $G$-equivariant morphism.
Let $H$ be a closed subgroup of $G$ and let $W$ be an $H$-invariant subvariety of $X$.
Assume that $G$ acts transitively on $Y$ and fix $y_0\in Y$. Let $G_0$ be the stabilizer of $y_0$ in $G$,
$X_0=p^{-1}(y_0)$ and $W_0=W\cap X_0$.
Assume that $X_0$ is irreducible and that $W=HW_0$.
Then, $G\cdot W=G\cdot W_0$ and
\[
G\cdot W\text{ is dense in $X$ if and only if $G_0\cdot W_0$ is dense in $X_0$}.
\]
\end{lemma}

\begin{proof}
The first assertion is clear since $W=HW_0$.
Suppose that $\cl{G_0\cdot W_0}=X_0$. Then $\cl{G\cdot W_0}=\cl{G\cdot G_0\cdot W_0}=\cl{G\cdot X_0}=X$
by the transitivity of $G$ on $Y$. Conversely, suppose that $\cl{G\cdot W_0}=X$.
Then $G\cdot W_0$ contains an open dense (and without loss of generality, $G$-invariant)
subset $U$ of $X$. The set
\[
\{y\in Y:p^{-1}(y)\cap U\ne\emptyset\}
\]
is $G$-invariant and non-empty (and in fact dense in $Y$). Since $G$ acts transitively on $Y$
we infer that $p^{-1}(y)\cap U\ne\emptyset$ for all $y\in Y$. In particular,
$X_0\cap U$ is a non-empty open subset of $X_0$, hence dense since $X_0$ is irreducible.
A fortiori, $X_0\cap G\cdot W_0$ is dense in $X_0$. However, $X_0\cap G\cdot W_0=G_0\cdot W_0$
since if $gx_0\in X_0$ where $x_0\in X_0$ and $g\in G$, then necessarily $g\in G_0$.
Hence, $G_0\cdot W_0$ is dense in $X_0$ as required.
\end{proof}

\subsection{}

We go back to the setup of \S\ref{sec: mm'}.
In order to invoke Lemma \ref{lem: svtrns} we will need an additional result.

\begin{lemma} \label{lem: elem}
We have the following equalities of spaces.
\begin{subequations}
\begin{gather} \label{eq: unipreporb}
\UnipAut\cdot\repn_{\n}=\overline{\OO_{\n}}\cap
(\rshft p^{\embd})^{-1}(\repn_{\m},\repn_{\m'})=\OO_{\n}\cap
(\rshft p^{\embd})^{-1}(\repn_{\m},\repn_{\m'}),\\ \label{eq: pareporb}
\ParAut\cdot\repn_{\n}=
\overline{\OO_{\n}}\cap (\rshft p^{\embd})^{-1}(\OO_\m\times\OO_{\m'})=
\OO_{\n}\cap (\rshft p^{\embd})^{-1}(\OO_\m\times\OO_{\m'}),\\ \label{eq: comporb}
\comp_\n\cap p_{\embd}^{-1}(\comp_{\m}^\circ\times\comp_{\m'}^\circ)=
\comp_\n^\circ\cap p_{\embd}^{-1}(\comp_{\m}^\circ\times\comp_{\m'}^\circ)=
\ParAut\cdot(\rshft p_{\tilde\Vect}^{-1}(\repn_{\n})\cap\commvar^{\embd}(\tilde\Vect)).
\end{gather}
\end{subequations}
A similar statement holds for $\embd'$.
\end{lemma}

\begin{proof}
By the $\ParAut$-equivariant of the map $\rshft p^{\embd}$
and the surjectivity of the map $\ParAut\rightarrow\Aut(\Vect)\times\Aut(\Vect')$, we have
\[
\ParAut\cdot((\rshft p^{\embd})^{-1}(\repn_{\m},\repn_{\m'}))=
(\rshft p^{\embd})^{-1}(\Aut(\Vect)\cdot \repn_{\m},\Aut(\Vect')\cdot \repn_{\m'})=(\rshft p^{\embd})^{-1}(\OO_\m\times\OO_{\m'}).
\]
Hence, the equality \eqref{eq: unipreporb} implies \eqref{eq: pareporb}

Next, we show that \eqref{eq: pareporb} implies \eqref{eq: comporb}.
Indeed, assuming \eqref{eq: pareporb}, we have
\[
\ParAut\cdot (\rshft p_{\tilde\Vect}^{-1}(\repn_{\n}))=
\rshft p_{\tilde\Vect}^{-1}(\ParAut\cdot\repn_{\n})=
\rshft p_{\tilde\Vect}^{-1}(\OO_{\n}\cap (\rshft p^{\embd})^{-1}(\OO_\m\times\OO_{\m'}))=
\comp_{\n}^\circ\cap (\rshft p^{\embd}\circ\rshft p_{\tilde\Vect})^{-1}(\OO_\m\times\OO_{\m'}).
\]
Since
\[
(\rshft p_{\Vect},\rshft p_{\Vect'})\circ p_{\embd}=\rshft p^{\embd}\circ\rshft p_{\tilde\Vect}\rest_{\commvar^{\embd}(\tilde\Vect)},
\]
we deduce that
\[
\ParAut\cdot(\rshft p_{\tilde\Vect}^{-1}(\repn_{\n}))\cap \commvar^{\embd}(\tilde\Vect)=
\comp_{\n}^\circ\cap p_{\embd}^{-1} (\rshft p_{\Vect}^{-1}(\OO_\m)\times\rshft p_{\Vect'}^{-1}(\OO_{\m'}))=
\comp_{\n}^\circ\cap p_{\embd}^{-1}(\comp_{\m}^\circ\times\comp_{\m'}^\circ).
\]
Also,
\begin{multline*}
\rshft p_{\tilde\Vect}(\comp_\n\cap p_{\embd}^{-1}(\comp_{\m}^\circ\times\comp_{\m'}^\circ))\subset
\rshft p_{\tilde\Vect}(\comp_\n)\cap
\rshft p_{\tilde\Vect}(p_{\embd}^{-1}(\comp_{\m}^\circ\times\comp_{\m'}^\circ))\subset
\\\overline{\OO_{\n}}\cap (\rshft p^{\embd})^{-1}
(\rshft p_{\Vect}(\comp_{\m}^\circ),\rshft p_{\Vect'}(\comp_{\m'}^\circ))=
\overline{\OO_{\n}}\cap (\rshft p^{\embd})^{-1}(\OO_\m\times\OO_{\m'})\subset\OO_{\n}
\end{multline*}
by \eqref{eq: pareporb} and hence
\[
\comp_\n\cap p_{\embd}^{-1}(\comp_{\m}^\circ\times\comp_{\m'}^\circ)\subset
\rshft p_{\tilde\Vect}^{-1}(\OO_{\n})=\comp_\n^\circ
\]
as required.

Note that the statements with respect to $\embd'$ are obtained from the original ones by interchanging
$\m$ and $\m'$. Thus, it remains to prove \eqref{eq: unipreporb}.

Since $\rshft p^{\embd}$ is $\UnipAut$-invariant,
the $\UnipAut$-orbit of $\repn_{\n}$ is contained in $\OO_{\n}\cap(\rshft p^{\embd})^{-1}(\repn_{\m},\repn_{\m'})$.

Next, we show that
\[
\overline{\OO_\n}\cap(\rshft p^{\embd})^{-1}(\repn_{\m},\repn_{\m'})\subset \OO_\n.
\]
Indeed, recall that
\[
\OO_\n=\{D\in \rshft \grend(\tilde\Vect):\rk D^r\rest_{\tilde\Vect_\rho}=
\rk\repn_{\n}^r\rest_{\tilde\Vect_\rho}\ \ \forall\rho\in\Cusp,\ r\ge0\}
\]
and
\[
\overline{\OO_\n}=\{D\in \rshft \grend(\tilde\Vect):\rk D^r\rest_{\tilde\Vect_\rho}\le
\rk\repn_{\n}^r\rest_{\tilde\Vect_\rho}\ \ \forall\rho\in\Cusp,\ r\ge0\}.
\]
On the other hand, for any $D\in (\rshft p^{\embd})^{-1}(\repn_{\m},\repn_{\m'})$ we have
\[
\rk D^r\rest_{\tilde\Vect_\rho}\ge \rk \repn_{\m}^r\rest_{\Vect_\rho}+\rk \repn_{\m'}^r\rest_{\Vect'_\rho}=
\rk\repn_{\n}^r\rest_{\tilde\Vect_\rho}.
\]
Our claim follows.

It remains to show that $\OO_{\n}\cap(\rshft p^{\embd})^{-1}(\repn_{\m},\repn_{\m'})$
is contained in the $\UnipAut$-orbit of $\repn_{\n}$.

Consider the abelian category whose objects are pairs $(U,D)$ where $U$ is a finite-dimensional
$\Cusp$-graded vector space and $D\in \rshft \grend(U)$; the morphisms between $(U,D)$ and
$(U',D')$ are the grading preserving linear transformations $T:U\rightarrow U'$ such that $T\circ D=D'\circ T$.

The statement that we need to prove is that if
\begin{equation} \label{eq: ses}
0\rightarrow (\Vect,\repn_{\m})\xrightarrow{\embd} (\tilde\Vect,D)\xrightarrow{\srjct} (\Vect',\repn_{\m'})\rightarrow0
\end{equation}
is a short exact sequence and $(\tilde\Vect,D)\simeq(\tilde\Vect,\repn_{\n})=(\Vect,\repn_{\m})\oplus(\Vect',\repn_{\m'})$,
then \eqref{eq: ses} splits. In fact, this is true for any locally finite $\C$-linear abelian category.
Indeed, we have an exact sequence
\[
0\rightarrow\Hom((\Vect',\repn_{\m'}),(\Vect,\repn_{\m}))\xrightarrow{\srjct_*}\Hom((\tilde\Vect,D),(\Vect,\repn_{\m}))
\xrightarrow{\embd_*}\Hom((\Vect,\repn_{\m}),(\Vect,\repn_{\m}))
\]
and
\[
\Hom((\tilde\Vect,D),(\Vect,\repn_{\m}))\simeq \Hom((\Vect',\repn_{\m'}),(\Vect,\repn_{\m}))
\oplus\Hom((\Vect,\repn_{\m}),(\Vect,\repn_{\m})).
\]
Comparing dimensions we infer that $\embd_*$ is onto. Hence, \eqref{eq: ses} splits.
\end{proof}

\subsection{Conclusion of proof of Propositions \ref{prop: LCint} and \ref{prop: LCint2}} \label{sec: concprof}

Using Lemmas \ref{lem: svtrns} and \ref{lem: elem} we show the equivalence of \eqref{cond: GVorb} and \eqref{cond: parorb}
with the first (resp., second) condition in Proposition \ref{prop: LCint2}, thereby completing the proofs
of Propositions \ref{prop: LCint} and \ref{prop: LCint2}.

For the first equivalence we apply Lemma \ref{lem: svtrns} to
\[
W=\comp_{\n}^\circ\cap p_{\embd'}^{-1}(\comp_{\m'}^\circ\times\comp_{\m}^\circ)\hookrightarrow
X=\comp_{\n}^\circ\xrightarrow{\rshft p_{\tilde\Vect}} Y=\OO_{\n}\ni y_0=\repn_{\n},
\]
with $G=\Aut(\tilde\Vect)$ and $H=\ParAutt$.
Note that $G_0=\stb_{\n}$ and the embedding $W_0\hookrightarrow X_0$ can be identified via $\lshft p_{\tilde\Vect}$ with
$\cntpress\hookrightarrow \cnt_{\n}$.
By  Lemma \ref{lem: elem} (with respect to $\embd'$) we have $W=H\cdot W_0$.

For the second equivalence, we apply Lemma \ref{lem: svtrns} with
\[
W=\dembed(\comp_{\m}^\circ\times\comp_{\m'}^\circ)\hookrightarrow
X=p_{\embd}^{-1}(\comp_{\m}^\circ\times\comp_{\m'}^\circ)\cap\comp_{\n}^{\circ}
\xrightarrow{\rshft p_{\tilde\Vect}}Y=\OO_{\n}\cap (\rshft p^{\embd})^{-1}(\OO_\m\times\OO_{\m'})\ni y_0=\repn_{\n},
\]
$G=\ParAut$ and $H=\Aut(\Vect)\times\Aut(\Vect')$.
By Lemma \ref{lem: elem}, $G$ acts transitively on $Y$. The condition $W=H\cdot W_0$ is clear.
We have $G_0=\unipstb$ and $W_0\hookrightarrow X_0$ can be identified using $\lshft p_{\tilde\Vect}$,
with $\cntbw\hookrightarrow\cntpres$. \qed

\begin{remark}
Let $p:\cntpres\rightarrow \cnt_{\m}\oplus \cnt_{\m'}$ be the restriction of $\lshft p^{\embd}$
to $\cntpres$. Thus, $p(\lambda)_{i,j}=\lambda_{i,j}$ for $(i,j)\in X_{\m}\cup X_{\m'}$.
Clearly, $p$ is $\unipstb$-invariant, to wit, $\unipstb$ acts on the fibers of $p$.

It is not difficult to see that $\LC(\m,\m')$ is also equivalent to the following conditions.
\begin{enumerate}
\item The set
\[
\{(\lambda,\lambda')\in \cnt_{\m}\oplus \cnt_{\m'}:p^{-1}(\lambda,\lambda')\text{ is a $\unipstb$-orbit}\}
\]
is dense in $\cnt_{\m}\oplus \cnt_{\m'}$.
\item The set
\[
\{(x,y)\in \comp_{\m}\times\comp_{\m'}:
\comp_{\n}\cap p_{\embd}^{-1}(\lambda,\lambda')\text{ is a $\UnipAut$-orbit}\}
\]
is dense in $\comp_{\m}\times\comp_{\m'}$.
\end{enumerate}
We will not give details since we will not use this result.
\end{remark}

\section{Corroborating evidence} \label{sec: mainev}

\subsection{}

For convenience, let us say that a multisegment $\m$ is \emph{good} if the conditions
$\LI(\m,\m')$ and $\LC(\m,\m')$ are equivalent for all $\m'\in\MS{\Seg}$.
Note that if $\m$ is good, then for any $\m'\in\MS{\Seg}$ the irreducibility of $Z(\m)\times Z(\m')$
is equivalent to $\IC(\m,\m')$.
By remark \ref{rem: LCsimcons}, Conjecture \ref{conj: weakform} is equivalent to saying that every $\m$ such that
$Z(\m)\in\IrrS$ is good, while Question \ref{qu: strongform} asks whether in fact every $\m\in\MS{\Seg}$ is good.

Recall that a multisegment $\m$ is called a (strict) ladder if it can be written as $\Delta_1+\dots+\Delta_k$ where
$\Delta_{i+1}\prec\Delta_i$ for all $i=1,\dots,k-1$. The corresponding representations were studied
in \cites{MR3163355, MR3573961}. In particular, $\GLS(\m)$ is satisfied and $Z(\m)$ is $\square$-irreducible
for any ladder $\m$.

The following result provides plenty of examples of good multisegments, in support of Conjecture \ref{conj: weakform}.
\begin{theorem} \label{thm: mainevid}
\begin{enumerate}
\item Every ladder multisegment is good.
\item Suppose that $\m_1,\dots,\m_k$ are good multisegments and that
$Z(\m_i)\times Z(\m_j)$ is irreducible for all $i,j$. Then $\m_1+\dots+\m_k$ is good.
\end{enumerate}
\end{theorem}

We note that if both $\m,\m'$ are ladders, then the condition $\LC(\m,\m')$ (and consequently,
the irreducibility of $Z(\m)\times Z(\m')$) admits a very simple
combinatorial description -- see \cite{MR3573961}*{Lemma 6.21}.

\begin{corollary}
If $Z(\m)$ is unitarizable, then $\m$ is good.
\end{corollary}

Indeed, it follows from Tadi\'c classification \cite{MR870688} (see also \cite{MR3163355})
that if $Z(\m)$ is unitarizable, then we can write $\m$ as a sum of multisegments
\[
\m=\m_1+\dots+\m_k
\]
where each $\m_i$ is a ladder of the form $\m_i=\Delta_{i,1}+\dots+\Delta_{i,r_i}$
where $\Delta_{i,j+1}=\lshft\Delta_{i,j}$, $j=1,\dots,r_i-1$,
and moreover $Z(\m_i)\times Z(\m_j)$ is irreducible for all $i,j$.

Another case which will be useful later is the following.

\begin{example}
Suppose that $\m=\Delta_1+\dots+\Delta_k$ and there exists $\rho\in\Cusp$ such that
$b(\Delta_i)\in\{\rho,\rshft\rho\}$ for all $i$.
Then, we can write $\m=\m_1+\dots+\m_l$ such that each $\m_i$ is a ladder (consisting
of either one or two segments) and $Z(\m_i)\times Z(\m_j)$ is irreducible for all $i,j$.
(See \S\ref{sec: onlyrhosrho} below.)
Thus, $\m$ is good.
\end{example}

We will prove the theorem (along with all other results stated in this section) in \S\ref{sec: proofs} below.

\subsection{}

The proof of Theorem \ref{thm: mainevid} depends on several compatibility properties of the conditions $\LI(\m,\m')$ and
$\LC(\m,\m')$ which are interesting in their own right.
We formulate them in the following propositions, which provide additional attestation for Conjecture \ref{conj: weakform} as well as for
an affirmative answer to Question \ref{qu: strongform}.

\subsubsection{}
For convenience (although not absolutely necessary) we fix a total order $\le$ on $\Cusp$ subject only to the condition that
$\rho<\rshft\rho$ for all $\rho\in\Cusp$.
For any non-zero multisegment $\m$ define $\max\m:=\max\supp\m$.
We will also define $\m^-$ in \eqref{eq: defm-} below.

\begin{proposition} \label{prop: mm-}
Suppose that $0\ne\m,\m'\in\MS{\Seg}$ with $\max\m<\max\m'$. Then
\begin{enumerate}
\item $\LI(\m,\m')$ if and only if $\LI(\m,\m'^-)$ and $(\m+\m')^-=\m+\m'^-$.
\item $\LC(\m,\m')$ if and only if $\LC(\m,\m'^-)$ and $(\m+\m')^-=\m+\m'^-$.
\end{enumerate}
Thus, Conjecture \ref{conj: weakform} and Question \ref{qu: strongform}
reduce to the case where $\max\m'\le\max\m$.
\end{proposition}

\subsubsection{}

The first part of Theorem \ref{thm: mainevid} will follow from Proposition \ref{prop: mm-} and the following.

\begin{proposition} \label{prop: splitdisj}
Let $\m=\sum_{i\in I}\Delta_i,\m'=\sum_{i'\in I'}\Delta_{i'}\in\MS{\Seg}$.
Assume that $I=I_1\cup I_2$ and $I'=I_1'\cup I_2'$ (disjoint unions) and that
\begin{subequations}
\begin{gather}
\label{cond: np1}\text{for any $i\in I_1\cup I_1'$ and $j\in I_2\cup I_2'$ we have $\Delta_i\not\prec\Delta_j$,}\\
\label{cond: np2}\text{for any $i\in I_1\cup I_1'$ and $j\in I_2\cup I_2'$ we have $\lshft\Delta_i\not\prec\Delta_j$.}
\end{gather}
\end{subequations}
Let $\m_r=\sum_{i\in I_r}\Delta_i$ and $\m'_r=\sum_{i'\in I'_r}\Delta_{i'}$, $r=1,2$. Then,
\begin{enumerate}
\item $\LC(\m,\m')$ implies $\LC(\m_1,\m_1')$ and $\LC(\m_2,\m_2')$.
\item $\LI(\m,\m')$ implies $\LI(\m_1,\m_1')$ and $\LI(\m_2,\m_2')$.
\end{enumerate}
\end{proposition}

Extending the partial order $\Cusp$ lexicographically to $\Seg$ (see \S\ref{sec: partorder}) we infer

\begin{corollary} \label{cor: gedelta}
For any two multisegments $\m=\sum_{i\in I}\Delta_i$, $\m'=\sum_{i'\in I'}\Delta_{i'}$ and $\Delta\in\Seg$ we have
\begin{enumerate}
\item $\LI(\m,\m')\implies\LI(\sum_{i\in I:\Delta_i\ge\Delta}\Delta_i,\sum_{i'\in I':\Delta_{i'}\ge\Delta}\Delta_{i'})$.
\item $\LC(\m,\m')\implies\LC(\sum_{i\in I:\Delta_i\ge\Delta}\Delta_i,\sum_{i'\in I':\Delta_{i'}\ge\Delta}\Delta_{i'})$.
\item $\LI(\m,\m')\implies\LI(\sum_{i\in I:e(\Delta_i)\in\Delta}\Delta_i,\sum_{i'\in I':e(\Delta_{i'})\in\Delta}\Delta_{i'})$.
\item $\LC(\m,\m')\implies\LC(\sum_{i\in I:e(\Delta_i)\in\Delta}\Delta_i,\sum_{i'\in I':e(\Delta_{i'})\in\Delta}\Delta_{i'})$.
\item $\LI(\m,\m')\implies\LI(\sum_{i\in I:b(\Delta_i)\in\Delta}\Delta_i,\sum_{i'\in I':b(\Delta_{i'})\in\Delta}\Delta_{i'})$.
\item $\LC(\m,\m')\implies\LC(\sum_{i\in I:b(\Delta_i)\in\Delta}\Delta_i,\sum_{i'\in I':b(\Delta_{i'})\in\Delta}\Delta_{i'})$.
\end{enumerate}
\end{corollary}

\subsubsection{}

The second part of Theorem \ref{thm: mainevid} follows from the following more precise statement.

\begin{proposition} \label{prop: 3ms}
Let $\m,\m',\n\in\MS{\Seg}$. Then,
\begin{enumerate}
\item $\LI(\m,\m')$ and $\LI(\m+\m',\n)$ imply $\LI(\m,\m'+\n)$.
\item $\LC(\m,\m')$ and $\LC(\m+\m',\n)$ imply $\LC(\m,\m'+\n)$ and $\LC(\m,\n)$.
\item $\LC(\m,\m')$ and $\LC(\m,\n)$ imply $\LC(\m,\m'+\n)$.
\item $\LC(\m,\n)$ and $\LC(\m',\n)$ imply $\LC(\m+\m',\n)$.
\item Suppose that $\LC(\m,\m')$ and $\LC(\m',\m)$.
Then
\[
\LC(\m+\m',\n)\iff\LC(\m,\n)\text{ and }\LC(\m',\n).
\]
\item Suppose that $Z(\m)$ and $Z(\m')$ are $\square$-irreducible
and $Z(\m)\times Z(\m')$ is irreducible. Then
\[
\LI(\m+\m',\n)\iff\LI(\m,\n)\text{ and }\LI(\m',\n).
\]
\end{enumerate}
\end{proposition}

Note the asymmetry in the statements of the proposition.
It would be interesting to show the missing representation-theoretic counterparts, namely
\begin{gather*}
\LI(\m,\m')\text{ and }\LI(\m+\m',\n)\implies\LI(\m,\n),\\
\LI(\m,\m')\text{ and }\LI(\m,\n)\implies\LI(\m,\m'+\n),\\
\LI(\m,\n)\text{ and }\LI(\m',\n)\implies\LI(\m+\m',\n),
\end{gather*}
at least when $Z(\m)$, $Z(\m')$ and $Z(\n)$ are $\square$-irreducible.

\subsection{}

The following compatibility result is motivated by \eqref{eq: factsirrs}.

\begin{proposition} \label{prop: sumofseg}
Let $\m,\m'\in\MS{\Seg}$ and $\n=\m+\m'$.
\begin{enumerate}
\item Suppose that $Z(\m)$, $Z(\m')$ are $\square$-irreducible and the conditions $\LI(\m,\m')$ and $\LI(\n,\m)$ are satisfied.
Then, $Z(\n)$ is $\square$-irreducible.
\item Suppose that the conditions $\GLS(\m)$, $\GLS(\m')$, $\LC(\m,\m')$ and $\LC(\n,\m)$ are satisfied.
Then, $\GLS(\n)$ holds.
\end{enumerate}
\end{proposition}

\subsection{}

For the last consistency check that we will state here, we fix $\rho\in\Cusp$.
For $\pi\in\Irr$ we write $\rho\sprt\pi$ if there does not exist $\pi'\in\Irr$ such that
$\pi\hookrightarrow\rho\times\pi'$. (A more general notation will be introduced in \S\ref{sec: pitchfork} below.)
For any $\pi=Z(\m)\in\Irr$ there exist a unique integer $m\ge0$ and a unique $\pi'=\Z(\m')\in\Irr$ such that
$\pi\hookrightarrow\overbrace{\rho\times\dots\times\rho}^m\times\pi'$ and $\rho\sprt\pi'$.
We denote $\lderiv\rho\m=\m'$.

\begin{proposition} \label{prop: rhoext}
Let $\m,\m'\in\MS{\Seg}$, $\rho\in\Cusp$ and assume that $\rho\sprt Z(\m')$.
Let $\tilde X^\rho_{\m,\m'}$ and $\tilde Y^\rho_{\m,\m'}$ be as in \eqref{def: xmm'rho} below.
Then
\begin{enumerate}
\item $\LI(\m,\m')\iff\LI(\lderiv\rho\m,\m')\text{ and }\#\tilde X^\rho_{\m,\m'}=\#\tilde Y^\rho_{\m,\m'}$.
\item $\LC(\m,\m')\iff\LC(\lderiv\rho\m,\m')\text{ and }\#\tilde X^\rho_{\m,\m'}=\#\tilde Y^\rho_{\m,\m'}$.
\end{enumerate}
\end{proposition}

\section{Some preparation} \label{sec: prep}

\subsection{}
For any $n,m\ge0$ let $J_{n,m}:\Reps_{n+m}\rightarrow\Reps_n\otimes\Reps_m$
be the Jacquet functor with respect to the standard parabolic subgroup of type $(n,m)$
where we identify $\Reps_n\otimes\Reps_m$ with the category of finite-length representations
of $\GL_n(F)\times\GL_m(F)$.\footnote{The tensor product of categories was defined by Deligne in \cite{MR1106898}.}
We obtain the ``total'' Jacquet functor
\[
J=\oplus_{n,m\ge0}J_{n,m}:\Reps\rightarrow\Reps\otimes\Reps.
\]
On the level of Grothendieck group this gives rise to a ring homomorphism
\[
J_{\ssf}:\Gr\rightarrow\Gr\otimes\Gr.
\]
For instance, for any segment $\Delta=\{\rho_1,\dots,\rho_k\}$ with $\rho_{i+1}=\rshft\rho_i$, $i=1,\dots,k-1$ we have
\begin{subequations} \label{eq: JZL}
\begin{gather}
J(Z(\Delta))=\oplus_{l=0}^kZ(\{\rho_1,\dots,\rho_l\})\otimes Z(\{\rho_{l+1},\dots,\rho_k\}),\\
J(L(\Delta))=\oplus_{l=0}^kL(\{\rho_{l+1},\dots,\rho_k\})\otimes L(\{\rho_1,\dots,\rho_l\}).
\end{gather}
\end{subequations}

By Frobenius reciprocity, for any $\pi,\pi'\in\Reps$ we have a canonical functorial surjection \cite{MR0579172}
\[
\canJ_{\pi,\pi'}:J(\pi\times\pi')\rightarrow\pi\otimes\pi'.
\]

For any $\cs\in\MS{\Cusp}$ let $\Reps_{\cs}$ be the Serre subcategory of $\Reps$ consisting of representations
all of whose irreducible subquotients have supercuspidal support $\cs$.
We have
\begin{equation} \label{eq: cuspdecomp}
\Reps=\oplus_{\cs\in\MS{\Cusp}}\Reps_{\cs}
\end{equation}
and $\Reps_{\cs}\times\Reps_{\cs'}\subset\Reps_{\cs+\cs'}$ for any $\cs,\cs'\in\MS{\Cusp}$.
For any $\pi\in\Reps$ and $\cs\in\MS{\Cusp}$ we denote by $\pi_{\cs}$ the $\cs$-component of $\pi$
with the respect to the decomposition \eqref{eq: cuspdecomp}.
Thus,
\[
\pi=\oplus_{\cs\in\MS{\Cusp}}\pi_{\cs}.
\]
Similarly, for any $\Pi\in\Reps\otimes\Reps$ and $\cs,\cs'\in\MS{\Cusp}$ we denote by
$\Pi_{\cs\otimes\cs'}$ the $\cs\otimes\cs'$-component of $\Pi$ with the respect to the decomposition
\[
\Reps\otimes\Reps=\oplus_{\cs,\cs'\in\MS{\Cusp}}\Reps_{\cs}\otimes\Reps_{\cs'}.
\]

\begin{remark}
Let $\pi\in\Reps$, $\cs\in\MS{\Cusp}$ and $\tau\in\Reps_{\cs}$. Any morphism $p:\pi\rightarrow\tau$
factor through a morphism $\pi_{\cs}\rightarrow\tau$.
Suppose that $p$ is surjective. Then, the restriction $p^{\cs}$ of $p$ to $\pi_{\cs}$ is surjective.
Moreover, the following conditions are equivalent.
\begin{enumerate}
\item $p^{\cs}$ is an isomorphism.
\item $\pi_{\cs}\simeq\tau$.
\item $\JH(\pi_{\cs})=\JH(\tau)$.
\item $\JH(\pi_{\cs})\le\JH(\tau)$ in $\MS{\Irr}$.
\end{enumerate}
In this case, we say that $p$ is a \emph{component map}, or that (somewhat informally) $\tau$ ``is''
the $\cs$-component of $\pi$. Similar terminology will apply for $\Reps\otimes\Reps$ and $\cs\otimes\cs'$
with $\cs,\cs'\in\MS{\Cusp}$.
\end{remark}

\subsection{Separated representations}

\subsubsection{} \label{sec: pitchfork}
The following technical definition will be useful.

\begin{definition} \label{def: sprt}
Suppose that $\pi\in\Reps_{\cs}$ and $\pi'\in\Reps_{\cs'}$ for some $\cs,\cs'\in\MS{\Cusp}$.
We write $\pi\sprt\pi'$ and say that $\pi$ is left-separated from $\pi'$ if the
$\canJ_{\pi,\pi'}$ is a component map, i.e., if the following equivalent conditions are satisfied.
\begin{enumerate}
\item The map $\canJ_{\pi,\pi'}$ factors through an isomorphism
$J(\pi\times\pi')_{\cs\otimes\cs'}\rightarrow\pi\otimes\pi'$.
\item $J(\pi\times\pi')_{\cs\otimes\cs'}\simeq\pi\otimes\pi'$.
\item $\JH(J(\pi\times\pi')_{\cs\otimes\cs'})=\JH(\pi\otimes\pi')(=\JH(\pi)\otimes\JH(\pi'))$.
\item $\JH(J(\pi\times\pi')_{\cs\otimes\cs'})\le\JH(\pi\otimes\pi')$.
\end{enumerate}
\end{definition}

As the notation suggests, the relation $\sprt$ is not symmetric.

\begin{example}
Suppose that $\rho\in\Cusp$ and $\pi\in\Irr$.
Then
\begin{equation} \label{eq: rhosprtpi}
\rho\sprt\pi\iff\nexists\pi'\in\Irr\text{ such that }\pi\hookrightarrow\rho\times\pi'
\iff\nexists\pi'\in\Reps\text{ such that }\pi\hookrightarrow\rho\times\pi'.
\end{equation}
\end{example}

The following easy property will be used repeatedly.

\begin{lemma} \label{lem: subrep}
Suppose that $\pi\in\Reps_{\cs}$ and $\pi'\in\Reps_{\cs'}$ for some $\cs,\cs'\in\MS{\Cusp}$.
If $\pi\sprt\pi'$, then $\tau\sprt\tau'$ for any subquotient $\tau$ of $\pi$ and $\tau'$ of $\pi'$.
Conversely, if $\tau\sprt\tau'$ for every $\tau\in\JH(\pi)$ and $\tau'\in\JH(\pi')$, then $\pi\sprt\pi'$.
\end{lemma}

\begin{proof}
Suppose that we have a short exact sequence
\[
0\rightarrow\sigma\rightarrow\pi\rightarrow\tau\rightarrow0.
\]
Then, we have a commutative diagram
\[
\xymatrix{0 \ar[r] & J(\sigma\times\pi')_{\cs\otimes\cs'} \ar[d]^{\canJ_{\sigma,\pi'}^{\cs\otimes\cs'}} \ar[r] &
J(\pi\times\pi')_{\cs\otimes\cs'} \ar[d]^{\canJ_{\pi,\pi'}^{\cs\otimes\cs'}}
 \ar[r]  &J(\tau\times\pi')_{\cs\otimes\cs'}\ar[r] \ar[d]^{\canJ_{\tau,\pi'}^{\cs\otimes\cs'}} & 0\\
0\ar[r] & \sigma\otimes\pi' \ar[r] & \pi\otimes\pi' \ar[r] & \tau\otimes\pi' \ar[r] & 0}
\]
where the rows are exact and the vertical morphisms are surjective.
It follows that $\pi\sprt\pi'$ if and only if $\sigma\sprt\pi'$ and $\tau\sprt\pi'$.
A similar statement holds for a short exact sequence
\[
0\rightarrow\sigma'\rightarrow\pi'\rightarrow\tau'\rightarrow0.
\]
The lemma immediately follows.
\end{proof}


We will use two additional properties of the relation $\sprt$.

\subsubsection{}

\begin{lemma} \label{lem: SIspr}
Let $\pi,\pi'\in\Irr$ and suppose that $\pi\sprt\pi'$. Then, $\pi\times\pi'$ is \SI\ and
\[
J(\soc(\pi\times\pi'))_{\cs(\pi)\otimes\cs(\pi')}=\pi\otimes\pi'.
\]
Moreover, up to isomorphism, $\soc(\pi\times\pi')$ is the only irreducible subquotient $\sigma$ of $\pi\times\pi'$ such that $J(\sigma)_{\cs(\pi)\otimes\cs(\pi')}\ne0$.
\end{lemma}

\begin{proof}
If $\sigma$ is a non-zero subrepresentation of $\pi\times\pi'$, then by Frobenius reciprocity $J(\sigma)\twoheadrightarrow\pi\otimes\pi'$.
Since $\pi\sprt\pi'$ it follows that $J(\sigma)_{\cs(\pi)\otimes\cs(\pi')}=J(\pi\times\pi')_{\cs(\pi)\otimes\cs(\pi')}=\pi\otimes\pi'$.
The statements now follow from the exactness of $J$.
\end{proof}

\subsubsection{}

\begin{lemma} \label{lem: socprod}
Suppose that $\pi\in\Reps_{\cs}$, $\pi'\in\Reps_{\cs'}$ and $\pi\sprt\pi'$.
Then,
\[
\soc(\pi\times\pi')=\soc(\soc(\pi)\times\soc(\pi')).
\]
\end{lemma}

\begin{proof}
Clearly, $\soc(\soc(\pi)\times\soc(\pi'))\hookrightarrow\soc(\pi\times\pi')$.
Let $\tau$ be an irreducible subrepresentation of $\pi\times\pi'$ and suppose on the contrary that
$\tau$ is not a subrepresentation of $\soc(\pi)\times\soc(\pi')$.
Then $\tau\cap(\soc(\pi)\times\soc(\pi'))=0$, and by the exactness of $J$ we have $J(\tau)\cap J(\soc(\pi)\times\soc(\pi'))=0$.
In particular,
\begin{equation} \label{eq: onehand}
J(\tau)_{\cs\otimes\cs'}\cap J(\soc(\pi)\times\soc(\pi'))_{\cs\otimes\cs'}=0.
\end{equation}
On the other hand, we claim that
\begin{equation} \label{eq: otherhand}
J(\soc(\pi)\times\soc(\pi'))_{\cs\otimes\cs'}=\soc(J(\pi\times\pi')_{\cs\otimes\cs'}).
\end{equation}
Indeed, since $\pi\sprt\pi'$, the restriction of $\canJ_{\pi,\pi'}$ to $J(\pi\times\pi')_{\cs\otimes\cs'}$ is an isomorphism.
It is therefore enough to check that the image of $J(\soc(\pi)\times\soc(\pi'))_{\cs\otimes\cs'}$ under $\canJ_{\pi,\pi'}$ is
$\soc(\pi\otimes\pi')=\soc(\pi)\otimes\soc(\pi')$, and this holds since $\soc(\pi)\sprt\soc(\pi')$ by Lemma \ref{lem: subrep}.

It follows from \eqref{eq: onehand} and \eqref{eq: otherhand} that $J(\tau)_{\cs\otimes\cs'}=0$ since any non-zero
subrepresentation of $J(\pi\times\pi')_{\cs\otimes\cs'}$ must intersect its socle nontrivially.
However, by Frobenius reciprocity $\Hom(J(\tau),\pi\otimes\pi')\ne0$ and hence $J(\tau)_{\cs\otimes\cs'}\ne0$. We get a contradiction.
\end{proof}

\subsection{} \label{sec: partorder}
Recall that we fixed a total order $\le$ on $\Cusp$ subject only to the condition that $\rho<\rshft\rho$ for all $\rho\in\Cusp$.
We extend this lexicographically to a total order on $\Seg$ (also denoted by $\le$), that is, $\Delta\le\Delta'$
iff either $e(\Delta)<e(\Delta')$ or $e(\Delta)=e(\Delta')$ and $b(\Delta)\le b(\Delta')$.
Clearly, if $\Delta\prec\Delta'$ (resp., $\lshft\Delta\prec\Delta'$) then $\Delta<\Delta'$ (resp., $\Delta\le\Delta'$).
Note that if $e(\Delta)=e(\Delta')$ then $\Delta\le\Delta'$ if and only if $\Delta\supseteq\Delta'$.
(In order to avoid confusion, we will never consider the partial order on $\Seg$ defined by inclusion.)

If $\m=\sum_{i\in I}\Delta_i\in\MS{\Seg}$ and $\Delta\in\Seg$, then we write
$\m_{\ge\Delta}=\sum_{i\in I:\Delta_i\ge\Delta}\Delta_i$ and similarly for $\m_{<\Delta}$.

\begin{lemma} \label{lem: simpspr}
Let $\m=\sum_{i\in I}\Delta_i,\m'=\sum_{i'\in I'}\Delta_{i'}\in\MS{\Seg}$ with $I\cap I'=\emptyset$.
Suppose that $\lshft\Delta_i\not\prec\Delta_{i'}$ for all $i\in I$, $i'\in I'$.
(For instance, this holds if there exists $\Delta\in\Seg$ such that $\m=\m_{\ge\Delta}$ and $\m'=\m'_{<\Delta}$.)
Then, $\std{\m}\sprt\std{\m'}$ and $\cstd{\m'}\sprt\cstd{\m}$.
\end{lemma}

\begin{proof}
We show that $\std{\m}\sprt\std{\m'}$, the other part being similar.
Suppose on the contrary that this is not the case. Then, by the geometric lemma and \eqref{eq: JZL}
there exists for each $i\in I\cup I'$ a partition $\Delta_i=A_i\cup B_i$ such that
\begin{itemize}
\item $\rho_1<\rho_2$ whenever $\rho_1\in A_i$ and $\rho_2\in B_i$, $i\in I$.
In particular, $A_i,B_i\in\Seg\cup\{\emptyset\}$.
\item $0\ne\sum_{i\in I}\supp B_i=\sum_{i'\in I'}\supp A_{i'}$ and in particular $\cup_{i'\in I'}A_{i'}=\cup_{i\in I}B_i\ne\emptyset$.
\end{itemize}
Let
\[
\rho=\min\cup_{i\in I}B_i=\min\cup_{i'\in I'}A_{i'}
\]
and let $i\in I$ be such that $\rho\in B_i$.
Then, $e(\Delta_i)\in B_i$ and therefore there exists $i'\in I'$ such that $e(\Delta_i)\in A_{i'}$.
Thus, $b(\Delta_{i'})\in A_{i'}$ and we infer that $b(\Delta_{i'})\ge\rho$ by the minimality of $\rho$.
Hence, $b(\Delta_{i'})\in B_i$. It follows from \eqref{eq: JZL} that $\lshft\Delta_i\not\prec\Delta_{i'}$
in contradiction to our assumption.
\end{proof}

\begin{corollary} \label{cor: sprtdelta}
Suppose that $\pi\in\Reps_{\cs}$, $\pi'\in\Reps_{\cs'}$
and there exists $\Delta\in\Seg$ such that each irreducible subquotient $Z(\m)$ of $\pi$ (resp., $\pi')$
satisfies $\m=\m_{\ge\Delta}$ (resp., $\m=\m_{<\Delta}$). Then, $\pi\sprt\pi'$.
A similar conclusion holds if each irreducible subquotient $L(\m)$ of $\pi$ (resp., $\pi')$
satisfies $\m=\m_{<\Delta}$ (resp., $\m=\m_{\ge\Delta}$).
\end{corollary}

\subsection{M\oe glin--Waldspurger algorithm} \label{sec: Zeleinv}
We recall the recursive algorithm, due to M\oe glin--Waldspurger \cite{MR863522}, for the computation of
the involution (of sets) $\m\mapsto\m^{\#}$ of $\MS{\Seg}$ defined in \eqref{def: MW}.
Set $0^{\#}=0$.
For $0\ne\m=\sum_{i\in I}\Delta_i\in\MS{\Seg}$ define
\[
\max\m=\max_{i\in I}e(\Delta_i).
\]
Let $m\ge1$ and $i_1,\dots,i_m\in I$ be indices such that
\begin{subequations} \label{eq: i1im}
\begin{gather}
\max\m=e(\Delta_{i_1})\text{ and }\Delta_{i_1}\text{ is maximal with respect to this property,}\\
\begin{split}\Delta_{i_{j+1}}\prec\Delta_{i_j}, e(\Delta_i)=e(\lshft{\Delta}_{i_j})\text{ and }
\Delta_{i_{j+1}}\text{ is maximal}\\\text{with respect to these properties },\ j=1,\dots,m-1,\end{split}\\
\text{There is no $i\in I$ such that $\Delta_i\prec\Delta_{i_m}$ and $e(\Delta_i)=e(\lshft{\Delta}_{i_m})$.}
\end{gather}
\end{subequations}
We call $i_1,\dots,i_m$ \emph{leading indices} of $\m$.
(Note that $i_1,\dots,i_m$ are not uniquely determined, but $m$ and $\Delta_{i_1},\dots,\Delta_{i_m}$ are.)
Let $\maxdel(\m)$ be the segment $\{e(\Delta_{i_m}),\dots,e(\Delta_{i_1})\}$ and set
\begin{equation} \label{eq: defm-}
\m^-=\sum_{i\in I}\Delta'_i\text{ where }\Delta'_i=\begin{cases}\Delta_i^-&\text{if }i\in\{i_1,\dots,i_m\},\\
\Delta_i&\text{otherwise.}\end{cases}
\end{equation}
(We discard $\Delta'_i$ if it is the empty set.)
By \cite{MR863522}*{\S II.2}, $\m^\#=(\m^-)^\#+\maxdel(\m)$ and in particular, $\supp(\m)=\supp(\m^-)+\supp(\maxdel(\m))$.
Moreover, $\maxdel(\m)$ is the smallest segment of $\m^\#$ such that $e(\maxdel(\m))=\max\m=\max\m^\#$.

\section{Proofs} \label{sec: proofs}

In this section we will prove the statements of \S\ref{sec: mainev}.

For the properties pertaining to the geometric condition $\LC(\m,\m')$
we will mostly use the last and most tangible criterion of Proposition \ref{prop: LCint2}.
It is likely, however, that the proofs can be made more conceptual.

When writing a multisegment $\m=\sum_{i\in I}\m_i$ it will be sometimes convenient to allow $\Delta_i=\emptyset$.
These inconsequential indices will not have any effect on $\m$ or on the objects pertaining to it
(such as $X_\m$, $Y_\m$ of \S\ref{sec: expGLS}).
We will use this convention throughout.

\subsection{Proof of Proposition \ref{prop: splitdisj} and Corollary \ref{cor: gedelta}}

Let $\pi=Z(\m)$, $\pi'=Z(\m')$, $\pi_j=Z(\m_j)$, $\pi_j'=Z(\m'_j)$, $j=1,2$.
Then, by \eqref{eq: LIsimp} and \eqref{cond: np1}, $\pi=\soc(\pi_1\times\pi_2)$
and $\pi'=\soc(\pi'_1\times\pi'_2)$.
Thus $J(\pi)\twoheadrightarrow\pi_1\otimes\pi_2$ and $J(\pi')\twoheadrightarrow\pi'_1\otimes\pi'_2$.
By the geometric lemma, it follows that $J(\pi\times\pi')$ admits $\pi_1\times\pi'_1\otimes\pi_2\times\pi'_2$ as a subquotient.
On the other hand by \eqref{cond: np2} and Lemma \ref{lem: simpspr}, $\pi_1\times\pi_1'\sprt\pi_2\times\pi_2'$ and hence
\[
J(\pi_1\times\pi_1'\times\pi_2\times\pi_2')_{\cs(\pi_1)+\cs(\pi_1')\otimes\cs(\pi_2)+\cs(\pi_2')}=
\pi_1\times\pi_1'\otimes\pi_2\times\pi_2'.
\]
Since $\JH(\pi_1\times\pi_1'\times\pi_2\times\pi_2')=\JH(\pi_1\times\pi_2\times\pi_1'\times\pi_2')$,
we infer that
\[
J(\pi\times\pi')_{\cs(\pi_1)+\cs(\pi_1')\otimes\cs(\pi_2)+\cs(\pi_2')}=\pi_1\times\pi'_1\otimes\pi_2\times\pi'_2.
\]
Moreover, $Z(\m_1+\m'_1)\sprt Z(\m_2+\m'_2)$ (since $\pi_1\times\pi_1'\sprt\pi_2\times\pi_2'$) and $\LI(\m_1+\m_1',\m_2+\m_2')$
(by \eqref{cond: np1} and \eqref{eq: LIsimp}). Hence, by Lemma \ref{lem: SIspr}
\[
J(Z(\m+\m'))_{\cs(\pi_1)+\cs(\pi_1')\otimes\cs(\pi_2)+\cs(\pi_2')}=Z(\m_1+\m'_1)\otimes Z(\m_2+\m'_2).
\]

If $Z(\m+\m')\hookrightarrow\pi\times\pi'$, then it follows that
\[
Z(\m_1+\m'_1)\otimes Z(\m_2+\m'_2)\hookrightarrow \pi_1\times\pi'_1\otimes\pi_2\times\pi'_2.
\]
This prove the first part of Proposition \ref{prop: splitdisj}.

To prove the second part, let $\lambda\in \cnt_\m$ and $\lambda'\in \cnt_{\m'}$.
For $r=1,2$, let $\tilde\lambda\in \cnt_{\m_r}$ be such that $\tilde\lambda_{i,j}=\lambda_{i,j}$ for all $(i,j)\in X_{\m_r}$.
Define $\tilde\lambda'\in \cnt_{\m'_r}$ similarly.
Clearly, for any $(i,j)\in X_{\m_r,\m_r'}$, the $Y_{\m_r,\m_r'}$-coordinates of
$v_{i,j}^{\m,\m'}(\lambda,\lambda')$ coincide with those of $v_{i,j}^{\m_r,\m_r'}(\tilde\lambda,\tilde\lambda')$.
On the other hand, the assumptions \eqref{cond: np1} and \eqref{cond: np2} imply that the coordinates outside $Y_{\m_r,\m_r'}$
vanish. This clearly implies the second part of Proposition \ref{prop: splitdisj}.

Corollary \ref{cor: gedelta} is now an immediate consequence.

\subsection{Proof of first part of Proposition \ref{prop: mm-}}

For any $0\ne\m=\sum_{i\in I}\Delta_i\in\MS{\Seg}$ we write
$\m=\m_{\mxprt}+\m_{\nmxprt}$ where
$\m_{\mxprt}=\sum_{i\in I:e(\Delta_i)=\max\m}\Delta_i$ and $\m_{\nmxprt}=\sum_{i\in I:e(\Delta_i)<\max\m}\Delta_i$.
For consistency we also write $0_{\mxprt}=0_{\nmxprt}=0$.

\subsubsection{}
\begin{lemma} \label{lem: nnmx}
For any $\m\in\MS{\Seg}$, the representation $L(\m_{\nmxprt})\times L(\m_\mxprt)$ is \SI\ and we have
$L(\m)=\soc(L(\m_{\nmxprt})\times L(\m_\mxprt))$.
Moreover, suppose that $L(\m')$ is an irreducible subquotient of $L(\m_{\nmxprt})\times L(\m_\mxprt)$ such that
$\m'_\mxprt=\m_\mxprt$. Then, $\m'=\m$.
\end{lemma}
This follows from \eqref{eq: LIsimp} and Lemmas \ref{lem: SIspr} and \ref{lem: simpspr}.

\subsubsection{}
\begin{lemma} \label{lem: socsocsoc}
Suppose that $\one\ne\pi\in\Irr$ and $0\ne\m\in\MS{\Seg}$ with $\max\supp\pi<\max\m$.
Then
\begin{multline} \label{eq: socsocsoc}
\soc(\pi\times L(\m))=\soc(\pi\times\soc(L(\m_{\nmxprt})\times L(\m_{\mxprt})))\\=
\soc(\soc(\pi\times L(\m_{\nmxprt}))\times L(\m_{\mxprt}))=\soc(\pi\times L(\m_{\nmxprt})\times L(\m_{\mxprt})).
\end{multline}
In particular, for any $\m'\in\MS{\Seg}$
\[
L(\m')\hookrightarrow\pi\times L(\m)\iff\m'_{\mxprt}=\m_{\mxprt}\text{ and }L(\m'_{\nmxprt})\hookrightarrow\pi\times L(\m_{\nmxprt}).
\]
\end{lemma}

\begin{proof}
The first equality of \eqref{eq: socsocsoc} follows from Lemma \ref{lem: nnmx} while the last one follows from
Lemma \ref{lem: socprod}, Corollary \ref{cor: sprtdelta} and the condition on $\pi$.
It remains to show that any irreducible subrepresentation $\pi'$ of $\pi\times L(\m_{\nmxprt})\times L(\m_{\mxprt})$
is contained in $\pi\times L(\m)$.
Assume on the contrary that this is not the case. Then
$\pi'\hookrightarrow\pi\times L(\m'')$ for some $L(\m'')\le L(\m_{\nmxprt})\times L(\m_{\mxprt})$
with $\m''_{\mxprt}\ne\m_{\mxprt}$ (by the second part of Lemma \ref{lem: nnmx}).
On the other hand, by the last equality of \eqref{eq: socsocsoc} we have
$\pi'\hookrightarrow\tau\times L(\m_{\mxprt})$ for some irreducible $\tau\hookrightarrow\pi\times L(\m_{\nmxprt})$
and hence by Lemma \ref{lem: nnmx} we have $\m'_{\mxprt}=\m_{\mxprt}$ where we write $\pi'=L(\m')$.

Now, $\pi'\hookrightarrow\pi\times L(\m''_{\nmxprt})\times L(\m''_{\mxprt})$.
Hence, $\pi'\hookrightarrow\tau'\times L(\m''_{\mxprt})$
for some irreducible subquotient $\tau'$ of $\pi\times L(\m''_{\nmxprt})$.
It follows from Lemma \ref{lem: nnmx} that
$\m''_{\mxprt}=\m'_{\mxprt}=\m_{\mxprt}$. We get a contradiction.

The last part follows from \eqref{eq: socsocsoc} and Lemma \ref{lem: nnmx}.
\end{proof}

\begin{corollary} \label{cor: wdelta}
Let $\pi$ and $\m$ be as in Lemma \ref{lem: socsocsoc}.
Let $\Delta$ be a segment in $\m$ such that $e(\Delta)=\max\m$. Then, for any $\m'\in\MS{\Seg}$
\[
L(\m')\hookrightarrow\pi\times L(\m)\iff\Delta\text{ is contained in $\m'$ and }
L(\m'-\Delta)\hookrightarrow\pi\times L(\m-\Delta).
\]
\end{corollary}

Indeed, if $\max(\m-\Delta)\ne\max\m$, i.e., if $\m_{\mxprt}=\Delta$ then this is just Lemma \ref{lem: socsocsoc};
otherwise this follows from Lemma \ref{lem: socsocsoc} by applying it to both $\m$ and $\m-\Delta$
(with $\m'$ and $\m'-\Delta$ respectively).

\subsubsection{}
In fact, the following slightly more precise statement holds (although we shall not use it here).
\begin{corollary}
Under the conditions of Corollary \ref{cor: wdelta} we have
\[
\soc(\pi\times L(\m))=
\soc(\soc(\pi\times L(\m-\Delta))\times L(\Delta))=\soc(\pi\times L(\m-\Delta)\times L(\Delta)).
\]
\end{corollary}

\begin{proof}
Since
\[
\pi\times L(\m)\hookrightarrow
\pi\times L(\m-\Delta)\times L(\Delta)\hookrightarrow
\pi\times L(\m_{\nmxprt})\times L(\m_{\mxprt}-\Delta)\times L(\Delta)
\]
and $L(\m_{\mxprt})=L(\m_{\mxprt}-\Delta)\times L(\Delta)$,
the equality $\soc(\pi\times L(\m))=\soc(\pi\times L(\m-\Delta)\times L(\Delta))$ follows from Lemma \ref{lem: socsocsoc}
which also implies that
\[
\soc(\pi\times L(\m))=\soc(\soc(\pi\times L(\m_{\nmxprt}))\times L(\m_{\mxprt}-\Delta)\times L(\Delta)).
\]
Note that since $L(\m_{\mxprt})=L(\m_{\mxprt}-\Delta)\times L(\Delta)$ is $\square$-irreducible, for any $\sigma\in\Irr$ we have
\[
\soc(\sigma\times L(\m_{\mxprt}-\Delta)\times L(\Delta))=
\soc(\soc(\sigma\times L(\m_{\mxprt}-\Delta))\times L(\Delta)).
\]
The same is therefore true for any semisimple $\sigma$.
Since $\soc(\pi\times L(\m_{\nmxprt}))$ is semisimple (by definition of $\soc$) we infer that
\[
\soc(\pi\times L(\m))=\soc(\soc(\soc(\pi\times L(\m_{\nmxprt}))\times L(\m_{\mxprt}-\Delta))\times L(\Delta)).
\]
On the other hand, applying Lemma \ref{lem: socsocsoc} to $\m-\Delta$ (if $\max(\m-\Delta)=\max\m$) we have
\[
\soc(\pi\times L(\m-\Delta))=\soc(\soc(\pi\times L(\m_{\nmxprt}))\times L(\m_{\mxprt}-\Delta)).
\]
(If $\max(\m-\Delta)\ne\max\m$ then the last relation is trivial.)
Hence,
\[
\soc(\pi\times L(\m))=\soc(\soc(\pi\times L(\m-\Delta))\times L(\Delta)).
\]
The corollary follows.
\end{proof}

\subsubsection{}
We can now prove the first part of Proposition \ref{prop: mm-}, which is a strengthening of \cite{MR3573961}*{Lemma 4.16}.

Suppose first that $(\m+\m')^-=\m+\m'^-$. Then, by comparing $\supp$ of both sides, $\maxdel(\m+\m')=\maxdel(\m')$.
For simplicity write $\Delta=\maxdel(\m')$ and $\pi=Z(\m)$. By Corollary \ref{cor: wdelta},
\begin{multline*}
Z(\m+\m')=L((\m+\m')^\#)\le\soc(\pi\times L(\m'^\#))=\soc(\pi\times Z(\m'))\iff
\\Z(\m+\m'^-)=Z((\m+\m')^-)=L((\m+\m')^\#-\Delta)\le\soc(\pi\times L(\m'^\#-\Delta))=
\soc(\pi\times Z(\m'^-))
\end{multline*}
as required.

It remains therefore to show that $\LI(\m,\m')$ implies $(\m+\m')^-=\m+\m'^-$.
The argument is essentially in \cite{MR3573961}. For completeness we recall it.
Write $\m=\sum_{i\in I}\Delta_i$, $\m'=\sum_{i'\in I'}\Delta_{i'}$ and let $i'_1,\dots,i'_{m'}$ be leading indices of $\m'$.
Assume on the contrary that $(\m+\m')^-\ne\m+\m'^-$.
Then, there exist indices $l'\le m'$ and $l\in I$ such that $\Delta_l\prec\Delta_{i'_{l'}}$, $e(\Delta_l)=e(\lshft\Delta_{i'_{l'}})$
and either $l'=m'$ or $\Delta_l\subsetneq\Delta_{i'_{l'+1}}$.
By Corollary \ref{cor: gedelta} we may replace $\m$ and $\m'$ by $\m_{\ge\Delta_l}$ and $\m'_{\ge\Delta_l}$ respectively
since neither the condition $\LI(\m,\m')$ nor the assumption $(\m+\m')^-\ne\m+\m'^-$ is affected by this change.
Thus, we may assume without loss of generality that $\m=\m_{\ge\Delta_l}$ and $l'=m'$.
By Lemma \ref{lem: socsocsoc} (applied to $\m'^\#$), the condition $\LI(\m,\m')$ implies that
\[
\m'^\#_{\mxprt}=(\m+\m')^\#_{\mxprt}
\]
However, the right-hand side admits a segment which contains $e(\Delta_l)$ in its support while the left-hand side does not.
We obtain a contradiction.

\subsection{Proof of second part of Proposition \ref{prop: mm-}}

\subsubsection{}

\begin{lemma} \label{lem: vaninorb}
Let $\m=\sum_{i\in I}\Delta_i\in\MS{\Seg}$ and let $i_1,\dots,i_m$ be leading indices of $\m$. (See \S\ref{sec: Zeleinv}.)
Then, any nonempty open $G_\m$-invariant subset $A$ of $\cnt_\m$ contains
an element $\lambda$ whose coordinates satisfy $\lambda_{i,i_j}=\delta_{i,i_{j+1}}$ for all $j=1,\dots,m-1$ and
$i\in I$ such that $e(\Delta_i)=e(\Delta_{i_{j+1}})$.
\end{lemma}

\begin{proof}
We argue as in \cite{MR863522}.
We prove by induction on $l=0,\dots,m-1$ that there exists $\lambda\in A$ such
$\lambda_{i,i_j}=\delta_{i,i_{j+1}}$ for all $j=1,\dots,l$ and
$i\in I$ such that $e(\Delta_i)=e(\Delta_{i_{j+1}})$.
The base of the induction ($l=0$) is trivial. For the induction step, we assume that $0<l<m$ and the statement
holds for $l-1$. In addition, by openness, we may assume that $\lambda_{i_{l+1},i_l}\ne0$.
Take $g\in G_\m$ whose coordinates are given by
\[
g_{i,j}=\begin{cases}\lambda_{i,i_l}&\text{if }j=i_{l+1}\text{ and }e(\Delta_i)=e(\Delta_{i_{l+1}}),\\\delta_{i,j}&\text{otherwise,}\end{cases}
\ \ (i,j)\in Y_{\m}.
\]
($g$ is invertible since $\lambda_{i_{l+1},i_l}\ne0$.) Let $\tilde\lambda=g^{-1}\lambda g\in A$. We show that $\tilde\lambda$ satisfies
the required conditions for $l$. Suppose first that $j<l$. By assumption, for any $\rho\in\Delta_{i_j}$ we have
$\lambda \xbasis_{\rho,i_j}=\xbasis_{\lshft\rho,i_{j+1}}+\xi$
where the coordinates of $\xi$ with respect to $\xbasis_{\lshft\rho,i}$ vanish if $e(\Delta_i)=e(\Delta_{i_{j+1}})$.
The same condition holds for $g^{-1}\xi$. It follows that
\[
\tilde\lambda \xbasis_{\rho,i_j}=
g^{-1}\lambda g\xbasis_{\rho,i_j}=g^{-1}\lambda \xbasis_{\rho,i_j}=\xbasis_{\lshft\rho,i_{j+1}}+g^{-1}\xi
\]
and hence, $\tilde\lambda_{i,i_j}=\delta_{i,i_{j+1}}$ for all $i$ such that $e(\Delta_r)=e(\Delta_{i_{j+1}})$.
Also, we may write
\[
\lambda \xbasis_{\rho,i_l}=\sum_{i:\Delta_i\prec\Delta_{i_l}\text{ and }e(\Delta_i)=e(\Delta_{i_{l+1}})}\lambda_{i,i_l}\xbasis_{\lshft\rho,i}+\xi=
g\xbasis_{\lshft\rho,i_{l+1}}+\xi
\]
where the coordinates of $\xi$ with respect to $\xbasis_{\rho,i}$ vanish if $e(\Delta_i)=e(\Delta_{i_{l+1}})$.
Hence,
\[
\tilde\lambda \xbasis_{\rho,i_l}=g^{-1}\lambda \xbasis_{\rho,i_l}=\xbasis_{\lshft\rho,i_{l+1}}+g^{-1}\xi=\xbasis_{\lshft\rho,i_{l+1}}+\xi.
\]
It follows that $\tilde\lambda_{i,i_l}=\delta_{i,i_{l+1}}$ for all $i$ such that $e(\Delta_i)=e(\Delta_{i_{l+1}})$.
This concludes the induction step.
\end{proof}

\subsubsection{}
Similarly, we have the following.
\begin{lemma} \label{lem: vaninorb2}
Let $\m\in\MS{\Seg}$.
Then, any nonempty open $G_{\m^-}$-invariant subset $B$ of $\cnt_{\m^-}$ contains an element
$\lambda$ whose coordinates satisfy $\lambda_{i,j}=0$ for all $(i,j)\in X_{\m^-}\setminus X_\m$.
\end{lemma}

\begin{proof}
Let $i_1,\dots,i_m$ be leading indices of $\m$.
Note that
\[
X_{\m^-}\setminus X_{\m}=\{(i_j,i)\in X_{\m^-}:j=2,\dots,m\text{ and }e(\Delta_{i_j})=e(\Delta_i)\}.
\]
We show by descending induction on $l=m,\dots,1$ that there exists $\lambda\in B$ such that
$\lambda_{i_j,i}=0$ for all $j>l$ and $i$ such that $e(\Delta_{i_j})=e(\Delta_i)$.
Once again, the base of the induction ($l=m$) is trivial.

For the induction step we may assume that $\Delta_{i_{l-1}}$ is not a singleton (or equivalently, $(i_l,i_{l-1})\in X_{\m^-}$)
since otherwise, $(i_l,i)\notin X_{\m^-}$
for all $i$ such that $e(\Delta_{i_j})=e(\Delta_i)$. We may also assume that $\lambda\in B$ satisfies $\lambda_{i_l,i_{l-1}}\ne0$.
Consider the set
\[
I_l=\{i\in I:(i_l,i)\in X_{\m^-}\text{ and }e(\Delta_{i_l})=e(\Delta_i)\}.
\]
Note that it follows from the definition of $i_l$ that $(i_{l-1},i)\in Y_{\m^-}$ for all $i\in I_l$.
Let $g\in G_{\m^-}$ be such that $g_{i_{l-1},i}=-\lambda_{i_l,i_{l-1}}^{-1}\lambda_{i_l,i}$ for all $i\in I_l$
and $g_{i,j}=\delta_{i,j}$ if $j\ne I_l$ or $i\ne i_{l-1}$. Let $\tilde\lambda=g^{-1}\lambda g\in B$.
Then, for any $\rho\in\Delta_i$, $i\in I_l$ the coefficients of
\[
\lambda g\xbasis_{\rho,i}=\lambda(\xbasis_{\rho,i}-\lambda_{i_l,i_{l-1}}^{-1}\lambda_{i_l,i}\xbasis_{\rho,i_{l-1}})
\]
with respect to $\xbasis_{\lshft\rho,j}$ are $0$ for all $j$ such that $\Delta_j\ge\Delta_{i_l}$.
Hence, the coefficient of $\xbasis_{\lshft\rho,i_l}$ in $\tilde\lambda \xbasis_{\rho,i}$ is $0$ as well.
Thus, $\tilde\lambda_{i_l,i}=0$ for all $i\in I_l$.
Moreover, if $i\in I_j$ with $j>l$ then $\tilde\lambda \xbasis_{\rho,i}=g^{-1}\lambda \xbasis_{\rho,i}=\lambda \xbasis_{\rho,i}$
since $g$ fixes $\xbasis_{\lshft\rho,k}$ for all $k$ such that $\Delta_k\le\Delta_{i_l}$. It follows from the induction hypothesis
that $\tilde\lambda_{i_j,i}=\lambda_{i_j,i}=0$ for all $i\in I_j$. This completes the induction step.
\end{proof}

\subsubsection{}

Recall the sets $X_{\m,\m'}$ and $Y_{\m,\m'}$ defined in \S\ref{def: Xmm'}.

\begin{lemma} \label{lem: fmm'}
Suppose that $0\ne\m,\m'\in\MS{\Seg}$ with $\max\m<\max\m'$.
Write $\m=\sum_{i\in I}\Delta_i$, $\m'=\sum_{i\in I'}\Delta_{i'}$ and let $i'_1,\dots,i'_{m'}$ be leading indices of $\m'$.
Write $\m'^-=\sum_{i\in I'}\Delta'_{i'}$ where
\[
\Delta'_{i'}=\begin{cases}\Delta_{i'}^-&\text{if $i'=i'_{j'}$ for some $j'=1,\dots,m'$},\\
\Delta_{i'}&\text{otherwise.}\end{cases}
\]
Then, $X_{\m,\m'^-}\subseteq X_{\m,\m'}$ and $Y_{\m,\m'^-}\subseteq Y_{\m,\m'}$.
Let $\tilde X_{\m,\m'}=X_{\m,\m'}\setminus X_{\m,\m'^-}$ and $\tilde Y_{\m,\m'}=Y_{\m,\m'}\setminus Y_{\m,\m'^-}$
and let $f_{\m,\m'}:\tilde Y_{\m,\m'}\rightarrow\tilde X_{\m,\m'}$ be given by $(i,i'_{j'})\mapsto (i,i'_{j'-1})$.
Then, upon fixing a total order on $I$, the function $f_{\m,\m'}$ is strictly monotone increasing with respect to the lexicographic order,
and in particular injective.
Moreover, $f_{\m,\m'}$ is a bijection (i.e., onto) if and only if $(\m+\m')^-=\m+\m'^-$.
In particular, this is the case if $\#\tilde X_{\m,\m'}\le\#\tilde Y_{\m,\m'}$.
\end{lemma}

\begin{proof}
The fact that $X_{\m,\m'^-}\subseteq X_{\m,\m'}$ and $Y_{\m,\m'^-}\subseteq Y_{\m,\m'}$ is
clear. The function $f_{\m,\m'}$ is well-defined since $\max\m<\max\m'$.
Note that
\[
\tilde X_{\m,\m'}=\{(i,i'_{j'})\in X_{\m,\m'}:j'=1,\dots,m',\ e(\Delta_i)=e(\lshft\Delta_{i'_{j'}})\}
\]
and
\[
\tilde Y_{\m,\m'}=\{(i,i'_{j'})\in Y_{\m,\m'}:j'=2,\dots,m',\ e(\Delta_i)=e(\Delta_{i'_{j'}})\}.
\]
The monotonicity of $f_{\m,\m'}$ is obvious.
The equivalence of $(\m+\m')^-=\m+\m'^-$ and the surjectivity of $f_{\m,\m'}$ is easy to check
(cf., proof of \cite{MR3573961}*{Proposition 6.19}).
\end{proof}

\subsubsection{}
We are now ready to show the second part of Proposition \ref{prop: mm-},

let $i'_1,\dots,i'_{m'}$ be leading indices of $\m'$.

As in Lemma \ref{lem: fmm'} we view $X_{\m,\m'^-}$ and $Y_{\m,\m'^-}$ as subsets of $X_{\m,\m'}$ and $Y_{\m,\m'}$
respectively and let $\tilde X_{\m,\m'}=X_{\m,\m'}\setminus X_{\m,\m'^-}$ and $\tilde Y_{\m,\m'}=Y_{\m,\m'}\setminus Y_{\m,\m'^-}$.

Assume first that $\LC(\m,\m')$ holds.
Thus, $v_{i,j}^{\m,\m'}(\lambda,\lambda')$, $(i,j)\in X_{\m,\m'}$ are linearly independent in $\C^{Y_{\m,\m'}}$
for some $(\lambda,\lambda')\in \cnt_\m\times \cnt_{\m'}$ (see \eqref{eq: vectors}).
Let
\begin{gather*}
\tilde X_{\m'}=\{(i',i'_{j'})\in X_{\m'}:j'=1,\dots,m'-1, i'\ne i'_{j'+1}\text{ and }e(\Delta_{i'})=e(\Delta_{i'_{j'+1}})\},\\
\hat X_{\m'}=\{(i'_{j'},i')\in X_{\m'}:e(\Delta_{i'_{j'}})=b(\lshft\Delta_{i'})\}.
\end{gather*}
Note that $X_{\m'}\setminus X_{\m'^-}=\tilde X_{\m'}\cup\hat X_{\m'}$.
By Lemma \ref{lem: vaninorb} (applied to $\m'$) we can assume without loss of generality that $\lambda'_{i,j}=0$
for all $(i,j)\in\tilde X_{\m'}$.

Let $\pr:\C^{Y_{\m,\m'}}\rightarrow\C^{Y_{\m,\m'^-}}$ be the standard projection.
It is easy to check that the vanishing of $\lambda'_{i,j}$ for $(i,j)\in\tilde X_{\m'}$
(regardless of the values of $\lambda'_{i,j}$, $(i,j)\in\hat X_{\m'}$) implies that
$\pr(v_{i,j}^{\m,\m'}(\lambda,\lambda'))=0$ for all $(i,j)\in \tilde X_{\m,\m'}$. 
In particular, $\#\tilde X_{\m,\m'}\le\#\tilde Y_{\m,\m'}$.
By Lemma \ref{lem: fmm'} we infer that $(\m+\m')^-=\m+\m'^-$ and $\#\tilde X_{\m,\m'}=\#\tilde Y_{\m,\m'}$.
It also follows that $\pr(v_{i,j}^{\m,\m'}(\lambda,\lambda'))$, $(i,j)\in X_{\m,\m'^-}$ are linearly independent in
$\C^{Y_{\m,\m'^-}}$.
However, once again because of the condition on $\lambda'$, we have
$\pr(v_{i,j}^{\m,\m'}(\lambda,\lambda'))=v_{i,j}^{\m,\m'^-}(\lambda,\tilde\lambda')$ for all $(i,j)\in X_{\m,\m'^-}$ where
the coefficients of $\tilde\lambda'\in \cnt_{\m'^-}$ are given by
$\tilde\lambda'_{i,j}=\lambda'_{i,j}$ if $(i,j)\in X_{\m'^-}\cap X_{\m'}$ and $\tilde\lambda'_{i,j}=0$ if $(i,j)\in X_{\m'^-}\setminus X_{\m'}$.
Hence, $\LC(\m,\m'^-)$ holds.

Conversely, assume that $\LC(\m,\m'^-)$ and $(\m+\m')^-=\m+\m'^-$.
Let $(\lambda,\tilde\lambda')\in \cnt_\m\times \cnt_{\m'}$ be such that $v_{i,j}^{\m,\m'^-}(\lambda,\tilde\lambda')$, $(i,j)\in X_{\m,\m'^-}$
are linearly independent in $\C^{Y_{\m,\m'^-}}$, and
let $f=f_{\m,\m'}:\tilde Y_{\m,\m'}\rightarrow\tilde X_{\m,\m'}$ be the bijection defined in Lemma \ref{lem: fmm'}.
By Lemma \ref{lem: vaninorb2} we may assume without loss of generality that $\tilde\lambda'_{i,j}=0$ if
$(i,j)\in X_{\m'^-}\setminus X_{\m'}$.
We may also assume that $\tilde\lambda'_{i'_{j'+1},i'_{j'}}\ne0$ for all $j'=1,\dots,m'-1$ such that $\Delta_{i'_{j'}}$ is not a singleton.
Define $\lambda'\in \cnt_{\m'}$ by
\[
\lambda'_{i,j}=\begin{cases}\tilde\lambda'_{i,j}&(i,j)\in X_{\m'}\cap X_{\m'^-},\\0&(i,j)\in\tilde X_{\m'},\\
1&(i,j)\in\hat X_{\m'}.\end{cases}
\]
As before,
\[
\pr(v_{i,j}^{\m,\m'}(\lambda,\lambda'))=\begin{cases}v_{i,j}^{\m,\m'^-}(\lambda,\tilde\lambda')&(i,j)\in X_{\m,\m'^-},\\
0&(i,j)\in\tilde X_{\m,\m'}.\end{cases}
\]
To conclude $\LC(\m,\m')$ it suffices to show that $v_{i,j}^{\m,\m'}(\lambda,\lambda')$, $(i,j)\in\tilde X_{\m,\m'}$ are linearly independent.
Indeed, for any $(i,j)\in\tilde Y_{\m,\m'}$ the
$(i,j)$-th coefficient of $v_{f(i,j)}^{\m,\m'}(\lambda,\lambda')$ is $\lambda'_{j,j'}\ne0$ where $f(i,j)=(i,j')$
while the only other non-zero coefficients of $v_{i,j}^{\m,\m'}(\lambda,\lambda')$ $(i,j)\in\tilde X_{\m,\m'}$ can occur at coordinates
$(r,j)$ with $r<i$. Thus, upon ordering $\tilde X_{\m,\m'}$ and $\tilde Y_{\m,\m'}$ lexicographically
(with respect to a prescribed total order on $I$),
the square matrix formed by the $\tilde Y_{\m,\m'}$-coordinates of $v_{i,j}^{\m,\m'}(\lambda,\lambda')$, $(i,j)\in\tilde X_{\m,\m'}$
is upper triangular with non-zero diagonal entries.

\subsection{Proof of first part of Theorem \ref{thm: mainevid}}
Suppose that $\m=\Delta_1+\dots+\Delta_k$ with $\Delta_{i+1}\prec\Delta_i$ for all $i=1,\dots,k-1$.
We show by induction on $k$ that the conditions $\LI(\m,\m')$ and $\LC(\m,\m')$ are equivalent for any multisegment
$\m'=\sum_{i\in I'}\Delta_{i'}$. The case $k=0$ is trivial.
Assume that $k>0$ and the statement holds for $k-1$.
By Proposition \ref{prop: mm-} we may assume that $\max\m'\le\max\m$.
In this case, by \cite{MR3573961}*{Proposition 6.1} $\LI(\m,\m')$ is equivalent to $\LI(\m-\Delta_1,\m'_{<\Delta_1})$.
To complete the induction step we show that the conditions $\LC(\m,\m')$ and $\LC(\m-\Delta_1,\m'_{<\Delta_1})$ are also equivalent.

Note that $X_{\m,\m'}=X_{\m-\Delta_1,\m'_{<\Delta_1}}\cup X_{\m-\Delta_1,\m'_{\ge\Delta_1}}$,
$Y_{\m,\m'}=Y_{\m-\Delta_1,\m'_{<\Delta_1}}\cup Y_{\Delta_1,\m'_{\ge\Delta_1}}\cup Y_{\m-\Delta_1,\m'_{\ge\Delta_1}}$
and $X_{\m'}=X_{\m'_{<\Delta_1}}\cup X_{\m'_{<\Delta_1},\m'_{\ge\Delta_1}}$.

We compare the vectors $v_{i,j}^{\m,\m'}(\lambda,\lambda')\in\C^{Y_{\m,\m'}}$, $(i,j)\in X_{\m,\m'}$
for $(\lambda,\lambda')\in \cnt_\m\times \cnt_{\m'}$ and
$v_{i,j}^{\m-\Delta_1,\m'_{<\Delta_1}}(\tilde\lambda,\tilde\lambda')\in\C^{Y_{\m-\Delta_1,\m'_{<\Delta_1}}}$, $(i,j)\in X_{\m-\Delta_1,\m'_{<\Delta_1}}$
for $(\tilde\lambda,\tilde\lambda')\in \cnt_{\m-\Delta_1}\times \cnt_{\m'_{<\Delta_1}}$.

Suppose that $\lambda_{i,j}=\tilde\lambda_{i,j}$ for all $(i,j)\in X_{\m-\Delta_1}$
and $\lambda'_{i,j}=\tilde\lambda'_{i,j}$ for all $(i,j)\in X_{\m'_{<\Delta_1}}$.
Then, it is easy to check 
that for all $(i,j)\in X_{\m-\Delta_1,\m'_{<\Delta_1}}$, the $Y_{\m-\Delta_1,\m'_{<\Delta_1}}$-coordinates of
$v_{i,j}^{\m,\m'}(\lambda,\lambda')$ and $v_{i,j}^{\m-\Delta_1,\m'_{<\Delta_1}}(\tilde\lambda,\tilde\lambda')$ coincide,
while the coordinates of $v_{i,j}^{\m,\m'}(\lambda,\lambda')$ outside $Y_{\m-\Delta_1,\m'_{<\Delta_1}}$ all vanish.
It follows that $\LC(\m,\m')\implies\LC(\m-\Delta_1,\m'_{<\Delta_1})$.
For the converse implication, observe that if $(i,j)\in X_{\m-\Delta_1,\m'_{\ge\Delta_1}}$, then
$(i-1,j)\in Y_{\m,\m'_{\ge\Delta_1}}$, $(i,i-1)\in X_{\m}$
and the $(i-1,j)$-coordinate of vectors $v_{i,j}^{\m,\m'}(\lambda,\lambda')$ is $\lambda_{i,i-1}$.
Moreover, if the $(i',j')$-coordinate of $v_{i,j}^{\m,\m'}(\lambda,\lambda')$ is nonzero for some $(i',j')\in Y_{\m,\m'_{\ge\Delta_1}}$,
then $j'=j$ and $i'<i$.
By imposing the open condition $\lambda_{i,i-1}\ne0$ for all $(i,i-1)\in X_{\m}$ we see that $\LC(\m-\Delta_1,\m'_{<\Delta_1})\implies\LC(\m,\m')$.
This completes the induction step.

\begin{remark}
In \cite{MR3573961}*{Proposition 6.15} we also proved that if at least one of $\m$ and $\m'$ is a ladder then
$\LC(\m,\m')$ is equivalent to the condition that there exists an injective map $f:X_{\m,\m'}\rightarrow Y_{\m,\m'}$
such that for all $(i,i')\in X_{\m,\m'}$ either $f(i,i')=(i,j')$ with $\Delta_{j'}\prec\Delta_{i'}$ or $f(i,i')=(i,j)$ with $\Delta_i\prec\Delta_j$.
In general, the latter condition is strictly weaker than $\LC(\m,\m')$.
\end{remark}

\subsection{Proof of Proposition \ref{prop: 3ms} and second part of Theorem \ref{thm: mainevid}}

Suppose that $Z(\m+\m'+\n)\hookrightarrow Z(\m+\m')\times Z(\n)$ and $Z(\m+\m')\hookrightarrow Z(\m)\times Z(\m')$.
Then, $Z(\m+\m'+\n)\hookrightarrow Z(\m)\times Z(\m')\times Z(\n)$. Thus, there exists an irreducible subquotient
$\pi$ of $Z(\m')\times Z(\n)$ such that $Z(\m+\m'+\n)\hookrightarrow Z(\m)\times\pi$. Since
$Z(\m+\m'+\n)\leunq Z(\m)\times Z(\m')\times Z(\n)$ and $Z(\m+\m'+\n)\leunq Z(\m)\times Z(\m'+\n)$,
we infer that necessarily $\pi=Z(\m'+\n)$. Hence, $\LI(\m,\m'+\n)$.

Now suppose that $\LC(\m+\m',\n)$ holds.
As before, we use the criterion \ref{part: linind} of Proposition \ref{prop: LCint2}.
Let $\lambda\in \cnt_{\m+\m'}$ and $\lambda'\in \cnt_{\n}$
be such that $v_{i,j}^{\m+\m',\n}(\lambda,\lambda')$, $(i,j)\in X_{\m+\m',\n}$ are linearly independent in $\C^{Y_{\m+\m',\n}}$.
Denote by $\lambda_{\m}$ the element of $\cnt_{\m}$ whose coordinates coincide with the $X_{\m}$-coordinates of $\lambda$.
Similarly for $\lambda_{\m'}$. \label{sec: lambdam}
If $\LC(\m,\m')$ holds, then by Proposition \ref{prop: LCint2}, we may assume that $\lambda_{i,j}=0$ for all $(i,j)\in X_{\m,\m'}$.
(We may still assume that $(\lambda_{\m},\lambda_{\m'})\in \cnt_{\m}\times \cnt_{\m'}$ is generic.)
In this case, the $Y_{\m',\n}$-coordinates of $v_{i,j}^{\m+\m',\n}(\lambda,\lambda')$, $(i,j)\in X_{\m,\n}$ vanish.
On the other hand, the $Y_{\m,\n}$-coordinates of $v_{i,j}^{\m+\m',\n}(\lambda,\lambda')$,
$(i,j)\in X_{\m,\n}$ coincide with those of $v_{i,j}^{\m,\n}(\lambda_{\m},\lambda')$.
The condition $\LC(\m,\n)$ follows.

Consider $v_{i,j}^{\m,\m'+\n}(\lambda_{\m},\mu)$ where $\mu_{\m',\n}=\mu_{\n,\m'}=0$ and $\mu_{\n}=\lambda'$.
Then, for $(i,j)\in X_{\m,\n}$ the $Y_{\m,\n}$-coordinates of $v_{i,j}^{\m,\m'+\n}(\lambda_{\m},\mu)$ are the same as those of
$v_{i,j}^{\m+\m',\n}(\lambda,\lambda')$ while the $Y_{\m,\m'}$-coordinates vanish.
Similarly, for $(i,j)\in X_{\m,\m'}$, the $Y_{\m,\m'}$-coordinates of $v_{i,j}^{\m,\m'+\n}(\lambda_{\m},\mu)$ are the same as those of
$v_{i,j}^{\m,\m'}(\lambda_{\m},\mu_{\m'})$ while the $Y_{\m,\n}$-coordinates vanish.
The property $\LC(\m,\m')$ therefore implies $\LC(\m,\m'+\n)$.

For the third part, write $X_{\m'+\n}=X_{\m'}\cup X_{\n}\cup X_{\m',\n}\cup X_{\n,\m'}$
and take $\lambda\in \cnt_{\m'+\n}$ whose coordinates with respect to $X_{\m',\n}\cup X_{\n,\m'}$ vanish.

The fourth part is similar.

The fifth part is an immediate consequence.

To prove the last part of Proposition \ref{prop: 3ms}, recall that $\LI(\m+\m',\n)$
if and only if the image of the intertwining operator $R_{Z(\n),Z(\m+\m')}$ is $Z(\m+\m'+\n)$.
On the other hand, by assumption $Z(\m+\m')=Z(\m)\times Z(\m')$ and therefore
\[
R_{Z(\n),Z(\m+\m')}=(\id\times R_{Z(\n),Z(\m')})\circ(R_{Z(\n),Z(\m)}\times\id).
\]
The required equivalence follows from the fact that $Z(\m+\m'+\n)\leunq Z(\n)\times Z(\m+\m')$,
$\JH(Z(\n+\m)\times Z(\m'))$ and $\JH(Z(\n+\m')\times Z(\m))$ and that the images of
$R_{Z(\n),Z(\m+\m')}$, $R_{Z(\n),Z(\m)}$ and $R_{Z(\n),Z(\m')}$ are irreducible.

This finishes the proof of Proposition \ref{prop: 3ms}.

The second part of Theorem \ref{thm: mainevid} now follows by induction on $k$.




\subsection{Proof of Proposition \ref{prop: sumofseg}}

The first part is a special case of \eqref{eq: factsirrs}.
We prove the second part.
Observe that $X_{\n}=X_\m\cup X_{\m'}\cup X_{\m,\m'}\cup X_{\m',\m}$ and $Y_{\n}=Y_\m\cup Y_{\m'}\cup Y_{\m,\m'}\cup Y_{\m',\m}$.
Assume that $\lambda\in \cntpress$.
Let $\lambda_{\m}$ and $\lambda_{\m'}$ be as in \S\ref{sec: lambdam}.
Then, for any $(i,j)\in X_{\n}$
\begin{enumerate}
\item The $Y_{\n,\m}$-coordinates of $v_{i,j}^{\n}(\lambda)$
are $v_{i,j}^{\n,\m}(\lambda,\lambda_{\m})$ if  $(i,j)\in X_{\n,\m}$ and $0$ otherwise.
\item If $(i,j)\in X_{\n,\m'}$ then the $Y_{\m'}$-coordinates of $v_{i,j}^{\n}(\lambda)$
are $v_{i,j}^{\m'}(\lambda_{\m'})$ if $(i,j)\in X_{\m'}$ and $0$ otherwise.
\item If $(i,j)\in X_{\m,\m'}$ then the $Y_{\m,\m'}$-coordinates of $v_{i,j}^{\n}(\lambda)$ are $v_{i,j}^{\m,\m'}(\lambda_{\m},\lambda_{\m'})$.
\end{enumerate}

It follows from the assumptions $\GLS(\m')$ and $\LC(\m,\m')$ that in order to show $\GLS(\n)$, it is enough to know that
$v_{i,j}^{\n,\m}(\lambda,\lambda_{\m})$, $(i,j)\in X_{\n,\m}$ are linearly independent for generic $\lambda\in \cnt_{\m',\m}$.
By the condition $\LC(\n,\m)$, $v_{i,j}^{\n,\m}(\lambda,\lambda')$, $(i,j)\in X_{\n,\m}$ are linearly independent for generic
$(\lambda,\lambda')\in \cnt_{\n}\times \cnt_{\m}$.
Since this condition in $\lambda$ is invariant under the action of $G_{\n}$,
it follows from Proposition \ref{prop: LCint2} and the condition $\LC(\m,\m')$
that $v_{i,j}^{\n,\m}(\lambda,\lambda')$, $(i,j)\in X_{\n,\m}$ are linearly independent for generic
$(\lambda,\lambda')\in  \cntpress\times \cnt_{\m}$.
Since this condition in $\lambda'$ is invariant under the action of $G_\m$, the condition $\GLS(\m)$ (which says that $\cnt_{\m}$ admits an open $G_{\m}$-orbit)
guarantees that we can take $\lambda'=\lambda_{\m}$, as required.

\subsection{Proof of Proposition \ref{prop: rhoext}}
For this subsection fix $\rho\in\Cusp$.

\subsubsection{}
For any $\pi\in\Irr$ there exists an integer $m\ge0$ and $\pi'\in\Irr$ such that $\pi\hookrightarrow \rho^{\times m}\times\pi'$
and $\rho\sprt\pi'$.
Moreover, $m$ and $\pi'$ are unique.
(In fact, $m$ is the largest non-negative integer for which $\pi\hookrightarrow \rho^{\times m}\times\pi'$
for some $\pi'\in\Reps$.)
We write $\mult^\rho(\pi)=m$ and $\lderiv\rho\pi=\pi'$ (the left $\rho$-derivative of $\pi$).
If $\pi=Z(\m)$ and $\pi'=Z(\m')$ we also write $\mult^\rho(\m)=m$ and $\lderiv\rho\m=\m'$.

If $\pi$ is $\square$-irreducible then so is $\lderiv\rho\pi$ \cite{MR3866895}*{Corollary 2.13}.
Similarly, $\GLS(\m)\implies\GLS(\lderiv\rho\m)$ \cite{MR3866895}*{Lemma 4.17}.

We now recall some results from \cites{MR2306606, MR2527415} (see also \cites{MR3573961, MR3866895}).

Let $\m=\sum_{i\in I}\Delta_i\in\MS{\Seg}$. We say that two subsets $A$ and $B$ of $I$
are equivalent and write $A\sim B$ if $\sum_{i\in A}\Delta_i=\sum_{i\in B}\Delta_i$.
Let
\[
X^\rho_{\m}=\{i\in I:\rho\prec\Delta_i\}=\{i\in I:b(\Delta_i)=\rshft\rho\}
\]
and
\[
Y^\rho_{\m}=\{i\in I:\lshft\rho\prec\Delta_i\}=\{i\in I:b(\Delta_i)=\rho\}.
\]
A one-to-one relation $R\subset Y^\rho_{\m}\times X^\rho_{\m}$ between $Y^\rho_{\m}$ and $X^\rho_{\m}$ is called a \emph{$\rho$-matching}
(with respect to $\m$)
if $\Delta_i\prec\Delta_j$ for all $(i,j)\in R$. We think of $R$ as a partially defined bijection and we write $R(i)=j$ and $R^{-1}(j)=i$ if $(i,j)\in R$.
We write
\[
A(R)=\{i\in Y^\rho_{\m}:R(i)\text{ is not defined}\},\ B(R)=\{j\in X^\rho_{\m}:R^{-1}(j)\text{ is not defined}\}.
\]

Given two $\rho$-matchings $R_1$, $R_2$ we say that $R_2$ dominates $R_1$ if one of the following conditions holds.
\begin{enumerate}
\item $R_2\supset R_1$.
\item There exist $i,j\in Y^\rho_{\m}$ and $k\in X^\rho_{\m}$ such that $R_2\setminus R_1=\{(j,k)\}$,
$R_1\setminus R_2=\{(i,k)\}$ and $\Delta_i<\Delta_j$.
\item There exist $i\in Y^\rho_{\m}$ and $j,k\in X^\rho_{\m}$ such that $R_2\setminus R_1=\{(i,k)\}$,
$R_1\setminus R_2=\{(i,j)\}$ and $\Delta_k<\Delta_j$.
\end{enumerate}
The transitive closure of domination is a partial order on the set of $\rho$-matchings, which we denote by $\le$.
Clearly, if $R\le R'$ then $\sum_{i\in A(R')}\Delta_i\le\sum_{i\in A(R)}\Delta_i$
in the sense that for every $\Delta\in\Seg$, $\#\{i\in A(R'):\Delta_i\ge\Delta\}\le\#\{i\in A(R):\Delta_i\ge\Delta\}$.
Clearly, a $\rho$-matching $R$ is maximal (with respect to $\le$) if and only if whenever $\Delta_i\prec\Delta_j$ with
$(i,j)\in Y^\rho_{\m}\times X^\rho_{\m}$ (exactly) one of the following possibilities occurs.
\begin{enumerate}
\item $R(i)$ and $R^{-1}(j)$ are defined.
\item $R(i)$ is defined, $R^{-1}(j)$ is not defined and $\Delta_j\ge\Delta_{R(i)}$.
\item $R(i)$ is not defined, $R^{-1}(j)$ is defined and $\Delta_i\le\Delta_{R^{-1}(j)}$.
\end{enumerate}

We say that two $\rho$-matchings $R_1$ and $R_2$ are equivalent if
\[
\sum_{(i,j)\in R_1}(\Delta_i,\Delta_j)=\sum_{(i,j)\in R_2}(\Delta_i,\Delta_j)
\]
as elements of $\MS{\Seg\times\Seg}$.

Maximal $\rho$-matchings are not unique up to equivalence.
(For instance, we could take $\m=\Delta_1+\Delta_2+\Delta_3+\Delta_4$ such that
$X_\m^\rho=\{3,4\}$, $Y_\m^\rho=\{1,2\}$, $\Delta_1\ne\Delta_2$, $\Delta_3\ne\Delta_4$
and $\Delta_i\prec\Delta_j$ for $i=1,2$, $j=3,4$. Then both
$\{(1,3),(2,4)\}$ and $\{(1,4),(2,3)\}$ are maximal $\rho$-matchings.)
However, up to equivalence, $A(R)$ and $B(R)$ are independent of $R$ if $R$ is maximal. 
We denote them by $A^\rho_{\m}$ and $B^\rho_{\m}$. (Technically, they are only defined up to equivalence, but this will not matter in what follows.)
Moreover,
\begin{enumerate}
\item $\lderiv\rho\m=\sum_{i\in I}\Delta'_i$ where
\[
\Delta'_i=\begin{cases}^-\Delta_i&i\in A^\rho_{\m},\\
\Delta_i&\text{otherwise.}\end{cases}
\]
Thus, $\mult^\rho(\m)=\#A^\rho_{\m}$ and in particular,
$\rho\sprt Z(\m)$ if and only if $A^\rho_{\m}=\emptyset$.
\item If $B^\rho_{\m}=\emptyset$, then $\soc(\rho\times Z(\m))=Z(\m+\{\rho\})$.
Otherwise, fix an index $i_0\in B^\rho_{\m}$ for which $\Delta_{i_0}$ is maximal. Then
$\soc(\rho\times Z(\m))=Z(\m')$ where
$\m'=\sum_{i\in I}\Delta'_i$ with $\Delta'_i=\begin{cases}^+\Delta_i&\text{if }i=i_0\\\Delta_i&\text{otherwise.}\end{cases}$
\end{enumerate}

A maximal $\rho$-matching $R$ between $Y^\rho_{\m}$ and $X^\rho_{\m}$ is called \emph{best $\rho$-matching} if
$\nexists (i,j),(i',j')\in R$ such that $\Delta_i<\Delta_{i'}\prec\Delta_j<\Delta_{j'}$.
A best $\rho$-matching exists and is unique up to equivalence.
Moreover, a maximal $\rho$-matching $R$ is best if and only if the product of the representations $Z(\Delta_i)$,
$i\in A^\rho_{\m}\cup B^\rho_{\m}$ and $Z(\Delta_i+\Delta_j)$, $(i,j)\in R$ is irreducible, in which case it is equal to
$Z(\sum_{i\in I:b(\Delta_i)\in\{\rho,\rshft\rho\}}\Delta_i)$. \label{sec: onlyrhosrho}

In particular, if $b(\Delta_i)\in\{\rho,\rshft\rho\}$ for all $i$,
then it follows from Theorem \ref{thm: mainevid} that
\begin{multline} \label{eq: brhos}
\text{$\m$ is good. Moreover, for any $\m'$,}\\
\LI(\m,\m')\iff\LI(\Delta_i,\m')\ \forall i\in A^\rho_{\m}\cup B^\rho_{\m}
\text{ and }\LI(\Delta_i+\Delta_j,\m')\ \forall (i,j)\in R.
\end{multline}

It is useful to introduce another notion which is weaker than maximality.
\begin{definition}
We say that a $\rho$-matching $R$ for $\m$ is saturated if $X_{\m}\cap(A(R)\times B(R))=\emptyset$
and for every $(i,j)\in R$ and $i'\in A(R)$ such that $(i',j)\in X_{\m}$ we have $\Delta_{i'}\le\Delta_i$.
\end{definition}

\begin{lemma} \label{lem: satur}
Let $R$ be a saturated $\rho$-matching $R$ for $\m$. Then, for any $R'\ge R$, $R'$ is saturated and
$A(R)=A(R')$. Hence, $A(R)\sim A^\rho_{\m}$.
\end{lemma}

\begin{proof}
It is enough to show it when $R'$ dominates $R$.
By the saturation of $R$ it is clear that $R'$ is obtained from $R$ by replacing a certain pair $(i,j)\in R$
by $(i,j')\in X_\m$ where $j'\in B(R)$ and $\Delta_{j'}<\Delta_j$. Thus, $A(R)=A(R')$ and $B(R')=B(R)\cup\{j\}\setminus\{j'\}$.
Let us show that $R'$ is saturated. First, if $(i',j)\in X_{\m}$ with $i'\in A(R)$ then $\Delta_{i'}\le\Delta_i$
by saturation and hence $(i',j')\in X_{\m}$ (since $(i,j')\in X_\m$). Thus, $X_{\m}\cap(A(R')\times B(R'))=\emptyset$.
On the other hand, we cannot have $i'\in A(R)$ with $(i',j')\in X_{\m}$ since $R$ is saturated.
Therefore, the second condition for saturation for $R'$ follows from the saturation of $R$.
\end{proof}

\subsubsection{}

For any $\m'\in\MS{\Seg}$ let
\begin{equation} \label{def: xmm'rho}
\tilde X^\rho_{\m,\m'}=X_{\m,\m'}\cap (A^\rho_{\m}\times X^\rho_{\m'}),\ \
\tilde Y^\rho_{\m,\m'}=Y_{\m,\m'}\cap (A^\rho_{\m}\times Y^\rho_{\m'}).
\end{equation}
As usual, we write $\tilde X^\rho_{\m}=\tilde X^\rho_{\m,\m}$.

\begin{lemma} \label{lem: char=}
Suppose that $\rho\sprt Z(\m')$. Then, $\#\tilde X^\rho_{\m,\m'}\ge\#\tilde Y^\rho_{\m,\m'}$.
Moreover, equality holds if and only if
\begin{equation} \label{eq: socrho}
\soc(\rho^{\times\mult^\rho(\m)}\times Z(\lderiv\rho\m+\m'))=Z(\m+\m').
\end{equation}
\end{lemma}

\begin{proof}
Let $R$ and $R'$ be maximal $\rho$-matchings with respect to $\m$ and $\m'$ respectively. By assumption, $A^\rho_{\m'}=\emptyset$.
The function $f:\tilde Y^\rho_{\m,\m'}\rightarrow\tilde X^\rho_{\m,\m'}$ given by $f(i,j')=(i,R'(j'))$ is well-defined
and injective.

Since $R\cup R'$ is a $\rho$-matching with respect to $\m+\m'$,
we have $\mult^\rho(\m+\m')=\#A^\rho_{\m+\m'}\le\#A^\rho_{\m}=\mult^\rho(\m)$.
By the above descriptions of $\lderiv\rho\m$ and $\soc(\rho\times *)$, the condition \eqref{eq: socrho}
means that $A^\rho_{\m+\m'}\sim A^\rho_{\m}$. We show that this is equivalent to the surjectivity of $f$.
Suppose that $f$ is not onto and let $(i,j')\in\tilde X^\rho_{\m,\m'}$ be outside the image of $f$.
Then, either $j'\in B^\rho_{\m'}$, in which case $R\cup R'\cup \{(i,j')\}$ dominates $R\cup R'$,
or $\Delta_{i'}<\Delta_i$ where $i'=R'^{-1}(j')$, in which case $R\cup R'\cup\{(i,j')\}\setminus\{(i',j')\}$
dominates $R\cup R'$. In both cases $\sum_{i\in A^\rho_{\m+\m'}}\Delta_i<\sum_{i\in A^\rho_{\m}}\Delta_i$.
Conversely, suppose that $f$ is onto. Then, $R\cup R'$ is saturated (with respect to $\m+\m'$).
It follows from Lemma \ref{lem: satur} that $A^\rho_{\m+\m'}\sim A^\rho_{\m}$ as required.
\end{proof}

\subsubsection{}
The following result follows from Lemma \ref{lem: SIspr}.
\begin{lemma}
Suppose that $\pi\in\Irr$ with $\rho\sprt\pi$ and let $m\ge1$.
Then, $\rho^{\times m}\times\pi$ is \SI, $\mult^\rho(\soc(\rho^{\times m}\times\pi))=m$
and if $\pi'$ is any subquotient of $\rho^{\times m}\times\pi$ other than $\soc(\rho^{\times m}\times\pi)$
then $\mult^\rho(\pi')<m$.
\end{lemma}

\begin{corollary} \label{cor: socinstg}
Suppose that $\pi'\in\Irr$ with $\rho\sprt\pi'$. Then, for any $\pi\in\Irr$ we have
\begin{equation} \label{eq: socinstg}
\soc(\pi\times\pi')=\soc(\rho^{\times m}\times\lderiv\rho\pi\times\pi')=
\soc(\rho^{\times m}\times\soc(\lderiv\rho\pi\times\pi')).
\end{equation}
\end{corollary}

\begin{proof}
Since $\rho\sprt\lderiv\rho\pi$ and $\rho\sprt\pi'$, we have $\rho\sprt\lderiv\rho\pi\times\pi'$.
The second equality of \eqref{eq: socinstg} follows from Lemma \ref{lem: socprod}.
To prove the first equality suppose on the contrary that $\sigma$ is a subrepresentation of
$\rho^{\times m}\times\lderiv\rho\pi\times\pi'$ which is not a subrepresentation of $\pi\times\pi'$.
Then, $\sigma\hookrightarrow\tau\times\pi'$ for some irreducible subquotient $\tau$ of $\rho^{\times m}\times\lderiv\rho\pi$
other than $\pi$. But then, $\mult^\rho(\sigma)=\mult^\rho(\tau)<m$. On the other hand, we already know
by the second equality of \eqref{eq: socinstg} that $\sigma\hookrightarrow\rho^{\times m}\times\sigma'$ for some $\sigma'$.
We get a contradiction.
\end{proof}

\subsubsection{}

\begin{lemma} \label{lem: opvan}
Let $R$ be a $\rho$-matching for $\m$.
Then, any non-empty open $G_\m$-invariant subset $S$ of $\cnt_\m$ contains an element $\lambda$ such that
$\lambda_{i,R(j)}=\delta_{i,j}$ for all $(i,j)\in Y_{\m}\cap (Y^\rho_{\m}\times Y^\rho_{\m})$ such that $R(j)$ is defined.
\end{lemma}

\begin{proof}
Let $\{j_1,\dots,j_r\}$ be the domain of $R$ with $\Delta_{j_1}\ge\dots\ge\Delta_{j_r}$.
We show by induction on $l$ that we can find $\lambda\in S$ such that
$\lambda_{i,R(j_m)}=\delta_{i,j_m}$ for all $m=1,\dots,l$ and all $i\in Y^\rho_{\m}$ such that $(i,j_m)\in Y_\m$.
The case $l=0$ is trivial.
For the induction step, let $0<l\le r$ and assume that the statement holds for $l-1$.
By openness, we may assume in addition that $\lambda$ satisfies $\lambda_{j_l,R(j_l)}\ne0$.
Define $g\in G_{\m}$ by
\[
g_{i,j}=\begin{cases}\lambda_{i,R(j_l)}&j=j_l,\\
\delta_{i,j}&\text{otherwise,}\end{cases}\\(i,j)\in Y_\m.
\]
Note that $g$ is invertible because $\lambda_{j_l,R(j_l)}\ne0$.

Let $\tilde\lambda=g^{-1}\lambda g$. Then, $gx_{\rshft\rho,R(j_l)}=x_{\rshft\rho,R(j_l)}$ and
\[
\lambda x_{\rshft\rho,R(j_l)}=\sum_i\lambda_{i,R(j_l)}x_{\rho,i}=g x_{\rho,j_l}+\sum_{i:(i,j_l)\notin Y_\m}\lambda_{i,R(j_l)}x_{\rho,i}.
\]
Thus,
\[
\tilde\lambda x_{\rshft\rho,R(j_l)}=x_{\rho,j_l}+\sum_{i:(i,j_l)\notin Y_\m}\lambda_{i,R(j_l)}x_{\rho,i}
\]
It follows that $\tilde\lambda_{i,R(j_l)}=\delta_{i,j_l}$ for all $i\in Y^\rho_\m$ such that $(i,j_l)\in Y_\m$.
Now let $m<l$. Then, $gx_{\rshft\rho,R(j_m)}=x_{\rshft\rho,R(j_m)}$ and by induction hypothesis
\[
\lambda x_{\rshft\rho,R(j_m)}=x_{\rho,j_m}+\xi
\]
where the coordinate of $\xi$ with respect to $x_{\rho,i}$ is zero unless $i\notin Y^\rho_{\m}$ or $\Delta_i>\Delta_{j_m}$.
Since $\Delta_{j_l}\le\Delta_{j_m}$ it follows that
\[
\tilde\lambda x_{\rshft\rho,R(j_m)}=\lambda x_{\rshft\rho,R(j_m)}=x_{\rho,j_m}+\xi.
\]
This completes the induction hypothesis.
\end{proof}

\subsubsection{}
By passing to the contragredient we get
\begin{lemma} \label{lem: opvan2}
Let $R$ be a $\rho$-matching for $\m$.
Then, any non-empty open $G_\m$-invariant subset $S$ of $\cnt_\m$ contains an element $\lambda$ such that
$\lambda_{R^{-1}(i),j}=\delta_{i,j}$ for all $(i,j)\in Y_{\m}\cap (X^\rho_{\m}\times X^\rho_{\m})$ such that $R^{-1}(i)$ is defined.
\end{lemma}

\subsubsection{}
Finally, we can prove Proposition \ref{prop: rhoext}.

By Lemma \ref{lem: char=} and Corollary \ref{cor: socinstg}, if $\#\tilde X^\rho_{\m,\m'}=\#\tilde Y^\rho_{\m,\m'}$,
then the conditions $\LI(\lderiv\rho\m,\m')$ and $\LI(\m,\m')$ are equivalent.
We show that $\LI(\m,\m')$ implies $\#\tilde X^\rho_{\m,\m'}=\#\tilde Y^\rho_{\m,\m'}$. By Corollary \ref{cor: gedelta},
we may assume that $b(\Delta_i)\in\{\rho,\rshft\rho\}$ for all $i\in I\cup I'$.
If $\LI(\m,\m')$, then it follows from \eqref{eq: brhos} that $\LI(\Delta_i,\m')$ for every $i\in A^\rho_{\m}$.
Taking a best matching $R'$ for $\m'$, this means that $\LI(\Delta_i,\Delta_{j'})$ for every $j'\in B^\rho_{\m'}$
and $\LI(\Delta_i,\Delta_{i'}+\Delta_{j'})$ for every $(i',j')\in R'$. This exactly means that the function
$f$ defined in the proof of Lemma \ref{lem: char=} is surjective, i.e. that $\#\tilde X^\rho_{\m,\m'}=\#\tilde Y^\rho_{\m,\m'}$.

For the second part, note that $X_{\lderiv\rho\m,\m'}\subseteq X_{\m,\m'}$, $Y_{\lderiv\rho\m,\m'}\subseteq Y_{\m,\m'}$
and
\[
\tilde X^\rho_{\m,\m'}=X_{\m,\m'}\setminus X_{\lderiv\rho\m,\m'},\ \
\tilde Y^\rho_{\m,\m'}=Y_{\m,\m'}\setminus Y_{\lderiv\rho\m,\m'}.
\]
Also, we have
\[
X_\m\setminus X_{\lderiv\rho\m}=\tilde X^\rho_{\m,\m}\cup \hat X_{\m}
\text{ where }\hat X_{\m}=\{i\in I:e(\Delta_i)=\lshft\rho\}\times A^\rho_{\m},
\]
and
\[
X_{\lderiv\rho\m}\setminus X_\m=X_{\lderiv\rho\m}\cap(Y^\rho_{\m}\times A^\rho_{\m}).
\]

Suppose that $\LC(\m,\m')$.
Let $(\lambda,\lambda')\in \cnt_{\m}\times \cnt_{\m'}$ be such that $v_{i,j}^{\m,\m'}(\lambda,\lambda')$, $(i,j)\in X_{\m,\m'}$
are linearly independent in $\C^{Y_{\m,\m'}}$.
Note that the $Y_{\lderiv\rho\m,\m'}$-coordinates of $v_{i,j}^{\m,\m'}(\lambda,\lambda')$, $(i,j)\in X_{\m,\m'}$
are independent of the $\hat X_{\m}$-coordinates of $\lambda$.
By Lemma \ref{lem: opvan} (for $R$ maximal) we may assume that $\lambda_{i,j}=0$ for all $(i,j)\in\tilde X^\rho_{\m}$.
This guarantees that
\begin{equation} \label{eq: prp1}
\text{for every $(i,j)\in\tilde X^\rho_{\m,\m'}$, the $Y_{\lderiv\rho\m,\m'}$-coordinates of $v_{i,j}^{\m,\m'}(\lambda,\lambda')$
vanish.}
\end{equation}
Hence, $\#\tilde X^\rho_{\m,\m'}\le\#\tilde Y^\rho_{\m,\m'}$. Thus, Lemma \ref{lem: char=} $\#\tilde X^\rho_{\m,\m'}=\#\tilde Y^\rho_{\m,\m'}$.
The vanishing of $\lambda_{i,j}=0$ for $(i,j)\in\tilde X^\rho_\m$ also guarantees that
\begin{equation} \label{eq: prp2}
\forall (i,j)\in X_{\lderiv\rho\m,\m'}, \text{the $Y_{\lderiv\rho\m,\m'}$-coordinates of $v_{i,j}^{\m,\m'}(\lambda,\lambda')$ and
$v_{i,j}^{\lderiv(\m),\m'}(\tilde\lambda,\lambda')$ coincide}
\end{equation}
where $\tilde\lambda\in \cnt_{\lderiv\rho\m}$ is given by
\[
\tilde\lambda_{i,j}=\begin{cases}\lambda_{i,j}&(i,j)\in X_{\lderiv(\m)}\cap X_\m,\\
0&(i,j)\in X_{\lderiv(\m)}\setminus X_\m.\end{cases}
\]
Thus, $\LC(\m,\m')\implies\LC(\lderiv\rho\m,\m')$.

Conversely, suppose that $\LC(\lderiv\rho\m,\m')$ is satisfied and $\#\tilde X^\rho_{\m,\m'}=\#\tilde Y^\rho_{\m,\m'}$.
Let $(\tilde\lambda,\lambda')\in \cnt_{\lderiv\rho\m}\times \cnt_{\m'}$ be such that $v_{i,j}^{\lderiv\rho\m,\m'}(\tilde\lambda,\lambda')$,
$(i,j)\in X_{\lderiv\rho\m,\m'}$
are linearly independent in $\C^{Y_{\lderiv\rho\m,\m'}}$.
By Lemma \ref{lem: opvan2} (applied to $\lderiv\rho\m$ and a maximal $R$) we may assume that
$\tilde\lambda_{i,j}=0$ for $(i,j)\in X_{\lderiv\rho\m}\setminus X_{\m}$.
Define $\lambda\in \cnt_\m$ by
\[
\lambda_{i,j}=\begin{cases}\tilde\lambda_{i,j}&(i,j)\in X_\m\cap X_{\lderiv\rho\m},\\0&(i,j)\in X_\m\setminus X_{\lderiv\rho\m}.
\end{cases}
\]
Then, \eqref{eq: prp1} and \eqref{eq: prp2} are satisfied.
Consider the (square) matrix $M'(\lambda')$ of size $\#\tilde X^\rho_{\m,\m'}\times\#\tilde Y^\rho_{\m,\m'}$
formed by the $\tilde Y^\rho_{\m,\m'}$-coordinates of $v_{i,j}^{\m,\m'}(\lambda,\lambda')$,
$(i,j)\in\tilde X^\rho_{\m,\m'}$. It is independent of $\lambda$ and depends only on the
$X_{\m'}\cap (Y^\rho_{\m'}\times X^\rho_{\m'})$-coordinates of $\lambda'$ (in fact, it does not depend on the
$X_{\m'}\cap (Y^\rho_{\m'}\times B^\rho_{\m'})$-coordinates of $\lambda'$).
To conclude $\LC(\m,\m')$ it remains to show that $M'(\lambda)$ is non-singular for generic $\lambda'\in \cnt_{\m'}$,
or equivalently for $\lambda'\in \cnt_{\m'}$ of our choice.
As in the proof of Lemma \ref{lem: char=}, let $R'$ be a maximal $\rho$-matching with respect to $\m'$ and
let $f:\tilde Y^\rho_{\m,\m'}\rightarrow\tilde X^\rho_{\m,\m'}$ be the injective function given by $f(i,j')=(i,R'(j'))$.
Since $\#\tilde X^\rho_{\m,\m'}=\#\tilde Y^\rho_{\m,\m'}$, $f$ is onto. Taking $\lambda'\in \cnt_{\m'}$ such that
$\lambda'_{i',j'}=\delta_{j',R'(i')}$ for all $(i',j')\in X_{\m'}\cap (Y^\rho_{\m'}\times X^\rho_{\m'})$,
$M'(\lambda')$ becomes the permutation matrix representing $f$. Our claim follows.

\def\cprime{$'$} 

\end{document}